\documentclass[12pt,oneside,english]{amsart}
\usepackage[T1]{fontenc}
\usepackage[latin9]{inputenc}
\usepackage{geometry}
\usepackage{amssymb}
\usepackage[all]{xy}

\makeatletter
\numberwithin{equation}{section} %% Comment out for sequentially-numbered
\numberwithin{figure}{section} %% Comment out for sequentially-numbered
  \@ifundefined{theoremstyle}{\usepackage{amsthm}}{}
  \theoremstyle{plain}
  \newtheorem{thm}{Theorem}[section]
  \theoremstyle{remark}
  \newtheorem{rem}[thm]{Remark}
  \theoremstyle{definition}
  \newtheorem{defn}[thm]{Definition}
 \theoremstyle{definition}
  \newtheorem{example}[thm]{Example}
  \theoremstyle{plain}
  \newtheorem{fact}[thm]{Fact}
  \theoremstyle{plain}
  \newtheorem*{thm*}{Theorem}
  \theoremstyle{plain}
  \newtheorem{lem}[thm]{Lemma}

\usepackage[all]{xy}
\usepackage{color}
\usepackage{mathdots}
\usepackage{stmaryrd}
\usepackage{dsfont}

\makeatletter
\newcommand{\xyC}[1]{%
\makeatletter
\xydef@\xymatrixcolsep@{#1}
\makeatother
} % end of \xyC

\makeatletter
\newcommand{\xyR}[1]{%
\makeatletter
\xydef@\xymatrixrowsep@{#1}
\makeatother
} % end of \xyR

\makeatother

\usepackage{babel}

\begin{document}

\author{Thilo Henrich}

\address{ Mathematisches Institut, Universität Bonn, Endenicher Allee 60,
53115 Bonn, Germany }

\email{henrich@math.uni-bonn.de}

\title{Mutation classes of diagrams via infinite graphs}

\thanks{This work was partially supported by the DFG project Bu-1866/1-2. }

\begin{abstract}
We give a complete description of the cluster-mutation classes of
diagrams of Dynkin types $\mathbb{A},\mathbb{B},\mathbb{D}$ and of
affine Dynkin types $\mathbb{B}^{(1)},\mathbb{C}^{(1)},\mathbb{D}^{(1)}$
via certain families of diagrams. 
\end{abstract}
\maketitle

\section{Introduction}

\noindent In 2002 Fomin and Zelevinsky introduced the concept of
cluster algebras \cite{S.2002}, which has been elaborated considerably
in recent years. The theory has connections to many important areas
in mathematics such as dual canonical bases for quantum groups or
total positivity for algebraic groups, but also Poisson geometry and
Teichmüller theory. To get a good overview, we recommend \cite{Fomin}
and \cite{Kellera}. Cluster algebras are not constructed as algebras
given by generators and relations, but rather by an iterative process
called mutation. Therefore understanding mutation is essential for
the whole theory.

At each step of the mutation process one obtains a special kind of
weighted graph called a diagram. The set of diagrams obtained via
mutation from a given diagram is called a mutation class, and it is
an important problem to determine the mutation class of a diagram.
By a result of Buan and Reiten \cite{Reiten} it is known that the
cardinality of the mutation class of a diagram whose underlying undirected
graph is of (affine) Dynkin type is finite. It is natural to ask whether
these finite mutation classes have an explicit combinatorial description.

In this paper we give a precise description of the mutation classes
of diagrams whose underlying graph is of Dynkin type $\mathbb{A},\mathbb{B},\mathbb{D}$
or affine Dynkin type $\mathbb{B}^{(1)},\mathbb{C}^{(1)},\mathbb{D}^{(1)}$.
This description is given in terms of certain families of diagrams
defined via infinite graphs. Note that the result for type $\mathbb{A}$
is already well known by work of Buan and Vatne \cite{A.2008a} as
well as Seven \cite{Seven2007}. Moreover there is some overlap with
independently achieved results by Vatne \cite{Vatne} for type $\mathbb{D}$.
Our result is the following (cf. sections \ref{sec:(Affine)-Dynkin-Diagrams}
and \ref{sec:The--diagrams} for precise definitions):

\begin{thm}
\label{Theorem (mutation classes)} Let $\Gamma$ be a connected diagram.
Then $\Gamma$ is of mutation type 
\end{thm}
\begin{enumerate}
\item $\mathbb{A}$ \emph{if and only if it is an $A$ diagram.} 
\item \emph{$\mathbb{B}$ if and only if it is a $B$ diagram.} 
\item \emph{$\mathbb{D}$ if and only if it is a $D_{\star}$ diagram, where
$\star\in\bigl\{(\bigcirc,n)_{n\geq3},\square,\boxslash,\bot\bigr\}$.} 
\item \emph{$\mathbb{B}^{(1)}$ if and only if it is a $B_{\star,B},B_{\star\wedge\star'}$
diagram, where $\star\in\bigl\{(\bigcirc,n)_{n\geq3},\square,\boxslash,\bot\bigr\}$
and $\star'\in\{B,\overleftrightarrow{B}\}$, or a $B_{\square\wedge\boxslash}$
diagram.} 
\item \emph{$\mathbb{C}^{(1)}$ if and only if it is a $C_{B,B}$ or $C_{B\wedge B}$
diagram.} 
\item \emph{$\mathbb{D}^{(1)}$ if and only if it is a $D_{\star,\star'},D_{\star\vee\star',}D_{\star\wedge\star'}$
diagram, where $\star,\star'\in\bigl\{(\bigcirc,n)_{n\geq3},\square,$}
$\boxslash,\bot,\overleftrightarrow{\boxslash}\bigr\}$\emph{, or
a $D_{\boxtimes}$ diagram.} 
\end{enumerate}
Let us briefly comment on the work which has been done concerning
the question of describing the mutation classes dealt with in this
paper. Following the classification of cluster algebras of finite
type by Fomin and Zelevinsky \cite{S.2003}, efforts were made to
find an effective way of determining whether a given diagram is mutation
equivalent to one with underlying graph of Dynkin type or not. Seven
\cite{Seven2007} gave an answer in terms of so called 'minimal forbidden
subdiagrams', while Barot, Gei\ss{} and Zelevinsky \cite{M.toappear} reduced the problem to
the question of cyclical orientability. The former result gives a
satisfying answer in the case of non simply-laced diagrams, but if
the diagram is simply-laced and contains more than nine vertices,
the result by Seven produces a checklist of more than two hundred
single diagrams plus five families of diagrams. Still, this list of
diagrams is very interesting since it coincides with the Happel-Vossieck
list \cite{D.1983}, which characterizes all minimal representation
infinite algebras. These connections were investigated in \cite{A.2007}.
It would be interesting to see whether the families of graphs produced
in our paper occur in a different context as well. Finally, let us
mention that the affine case $\mathbb{A}^{(1)}$ can be dealt with
in a similar way to our results. However, additional numerics are
required in order to parametrize the combinatorics involved, cf. the
work by Bastian \cite{Bastian}.

One example for applications of our results is the derived equivalence
classification of cluster-tilted algebras (i.e. the endomorphism algebras
of cluster-tilting objects in the cluster-tilting category), cf. \cite{A.2006,A.2007a}.
Using their own results mentioned above, this was done for type $\mathbb{A}$
by Buan and Vatne \cite{A.2008a} and for type $\mathbb{A}^{(1)}$
by Bastian \cite{Bastian}. \\

\noindent The methods applied in this paper are purely combinatorial.
In section \ref{sec:(Affine)-Dynkin-Diagrams} we recall the definitions
of the (affine) Dynkin diagrams occurring in Theorem \ref{Theorem (mutation classes)}
and the notion of diagram mutation. Section \ref{sec:The--diagrams}
then lists, for each (affine) Dynkin type, a certain number of infinite
undirected graphs inducing the families of diagrams mentioned in the
Theorem. A summary is given at the end of that section. The final
section \ref{sec:Proof} contains the proof of Theorem \ref{Theorem (mutation classes)}.\\

\textbf{Acknowledgement.} The author would like to express his sincere
thanks to the referees of {}``Mathematische Nachrichten'' for detailed
proofreading of the manuscript and their helpful suggestions. Also,
we would like to thank Sefi Ladkani and Maurizio Martino for their
help in improving the presentation of the material.

\section{(Affine) Dynkin Diagrams and Mutation of Diagrams\label{sec:(Affine)-Dynkin-Diagrams}}

\subsection{Some (affine) Dynkin Diagrams}

Recall the definitions of the (affine) Dynkin diagrams mentioned in
Theorem \ref{Theorem (mutation classes)}, which are 

\begin{itemize}
\item the Dynkin diagrams of type $\mathbb{A},\mathbb{B},\mathbb{D}$: 
\end{itemize}
\[
\xyR{.5pc}\xymatrix{A_{n} & \circ\ar@{-}'[r]'[rr] & \circ & \circ\ar@{}[r]|{\cdots} & \circ\\
B_{n} & \circ\ar@{-}[r]^{2} & \circ\ar@{-}[r] & \circ\ar@{}[r]|{\cdots} & \circ\\
D_{n} & \circ\ar@{-}'[r]'[rr] & \circ & \circ\ar@{}[r]|{\cdots} & \circ\\
 &  & \circ\ar@{-}[u]}
\]

\begin{itemize}
\item the affine Dynkin diagrams of type $\mathbb{B}^{(1)},\mathbb{C}^{(1)},\mathbb{D}^{(1)}$
(we also include type $\mathbb{A}^{(1)}$, which will be used in the
definitions and proofs):\[
\xyR{.5pc}\xymatrix{A_{n}^{(1)} &  & \circ & \circ\\
 & \circ\ar@{-}'[ur]'[urr]'[rrr]'[drr] &  &  & \circ\\
 &  & \circ\ar@{-}[ul] & \circ\ar@{}[l]|{\cdots}}
\]

\end{itemize}
\[
\xyR{.5pc}\xymatrix{B_{n}^{(1)} & \circ\ar@{-}'[r]'[rr] & \circ & \circ\ar@{}[r]|{\cdots} & \circ\ar@{-}[r]^{2} & \circ\\
 &  & \circ\ar@{-}[u]\\
C_{n}^{(1)} & \circ\ar@{-}[r]^{2} & \circ\ar@{-}[r] & \circ\ar@{}[r]|{\cdots} & \circ\ar@{-}[r] & \circ\ar@{-}[r]^{2} & \circ\\
D_{n} & \circ\ar@{-}'[r]'[rr] & \circ & \circ\ar@{}[r]|{\cdots} & \circ\ar@{-}[r] & \circ\ar@{-}[r]^{} & \circ\\
 &  & \circ\ar@{-}[u] &  &  & \circ\ar@{-}[u]}
\]

\subsection{Some Notations}

A graph $\Gamma=(\Gamma_{0},\Gamma_{1})$ is always given by its set
of vertices $\Gamma_{0}$ and its set of edges $\Gamma_{1}$. The
cardinality of $\Gamma_{0}$ is denoted by $|\Gamma_{0}|$ and for
$x\in\Gamma_{0}$, $\mbox{deg}(x)$ is the cardinality of the set
of neighbours of $x$ in the underlying undirected graph of $\Gamma$.
If $\Gamma$ is directed we also distinguish between $\mbox{deg}^{-}(x)$,
the number of neighbours $y$ of $x$ in $\Gamma$ such that there
exists an element $\{y\rightarrow x\}$ in $\Gamma_{1}$, and $\mbox{deg}^{+}(x):=\mbox{deg}(x)-\mbox{deg}^{-}(x)$.

Given any subset $S\subseteq\Gamma_{0}$, the \emph{induced subgraph}
\emph{on} $S$ is the subgraph of $\Gamma$ obtained from $\Gamma$
by deleting all vertices in $\Gamma_{0}\setminus S$ and all edges
adjacent to these. A subgraph $\Gamma'\subseteq\Gamma$ is called
\emph{full}, if $\Gamma'$ is the induced subgraph on $\Gamma'_{0}$.
A directed graph $\Gamma$ is called \emph{cyclically oriented} if
for all full subgraphs $\Gamma'\subseteq\Gamma$ with underlying graph
a cycle (i.e. equal to $A_{|\Gamma'_{0}|}^{(1)}$), $\Gamma'$ is
oriented clockwise or anti-clockwise.

For any (affine) Dynkin type $\star$ and $\Gamma$ any directed graph,
$\Gamma$ is called a $\star$ graph if the underlying undirected
graph of $\Gamma$ is of type $\star$.

\subsection{Diagram mutation}

Let $\Gamma$ be a finite directed graph whose edges are weighted
with positive integers. $\Gamma$ is called a \emph{diagram} if the
product of weights along any cycle is an integer which is the square
of an integer. For any vertex $k\in\Gamma_{0}$, the \emph{mutation}
\emph{of $\Gamma$ in $k$}, denoted by $\mu_{k}$, is the following
transformation of $\Gamma$:

\begin{itemize}
\item the orientations of all edges incident to $k$ are reversed, their
weights remain unchanged. 
\item for any vertices $i$ and $j$ which are connected in $\Gamma$ via
a two-edge oriented path going through $k$ (refer to the figure below
for the rest of notation), the direction of the edge $(i,j)$ in $\mu_{k}(\Gamma)$
and its weight $c'$ are uniquely determined by the rule \[
\pm\sqrt{c}\pm\sqrt{c'}=\sqrt{ab}\]
 where the sign before $\sqrt{c}$ (respectively before $\sqrt{c'}$)
is $+$ if $i,j,k$ form an oriented cycle in $\Gamma$ (resp., in
$\mu_{k}(\Gamma)$), and is $-$ otherwise. Here either $c$ or $c'$
can be equal to $0$. \[
\xyR{.7pc}
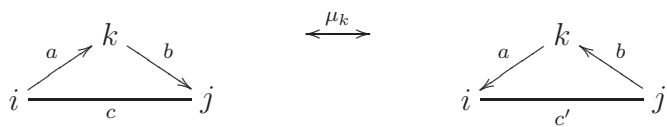
\xymatrix{ & k\ar[dr]^{b} &  & \ar@{<->}[r]^{\mu_{k}} &  &  & k\ar[dl]_{a}^{}\\
i\ar[ur]^{a} &  & j\ar@{-}[ll]^{c} &  &  & i &  & j\ar@{-}[ll]^{c'}\ar[ul]_{b}^{}}
\]

\item the rest of the edges and their weights in $\Gamma$ remain unchanged. 
\end{itemize}
Note that the resulting weighted graph is a diagram again. Two diagrams
$\Gamma$ and $\Gamma'$ related by a sequence of mutations $\mu_{i_{1}}\circ\mu_{i_{2}}\circ\ldots\circ\mu_{i_{n}}$
, where $\{i_{j}\}_{1\leq j\leq n}\subseteq\Gamma_{0}$, are called
\emph{mutation equivalent}. Since $\mu_{k}$ is involutive, i.e. $\mu_{k}^{2}(\Gamma)=\Gamma$
, this defines an equivalence relation on the set of all diagrams.
The equivalence classes are called \emph{mutation classes.} The task
of understanding these classes is not trivial at all. For example,
we shall see that diagram $i)$ below is neither mutation equivalent
to diagram $ii)$ nor to diagram $iii)$, while $ii)$ and $iii)$
belong to the same mutation class.\[
\xyC{1.5pc}\xyR{.4pc}
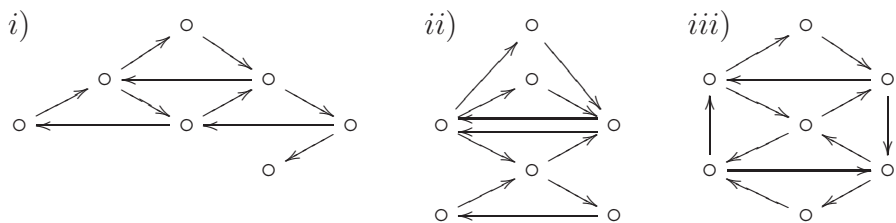
\xymatrix{i) &  & \circ\ar[dr] &  &  & ii) & \circ\ar[ddr] &  & iii) & \circ\ar[dr]\\
 & \circ\ar[dr]\ar[ur] &  & \circ\ar[ll]\ar[dr] &  &  & \circ\ar[dr] &  & \circ\ar[ur]\ar[dr] &  & \circ\ar[ll]\ar[dd]\\
\circ\ar[ur] &  & \circ\ar[ll]\ar[ur] &  & \circ\ar[ll]\ar[dl] & \circ\ar[dr]\ar[ur]\ar[uur] &  & \circ\ar@<.5ex>[ll]\ar@<-.5ex>[ll] &  & \circ\ar[ur]\ar[dl]\\
 &  &  & \circ &  &  & \circ\ar[dr]\ar[ur] &  & \circ\ar[rr]\ar[uu] &  & \circ\ar[ul]\ar[dl]\\
 &  &  &  &  & \circ\ar[ur] &  & \circ\ar[ll] &  & \circ\ar[ul]}
\]
It is well known that for diagrams with underlying undirected graph
being a tree, the mutation class is independent of the particular
orientation (cf. \cite{S.2003} Prop. 9.2). Hence we say that for
$\star$ an (affine) Dynkin type other than $\mathbb{A}^{(1)}$, a
diagram $\Gamma$ is of \emph{mutation type} $\star$ if it is mutation
equivalent to a diagram with underlying undirected graph of type $\star$.
Finally, a diagram is called \emph{simply-laced} if all its edge weights
are equal to one.

\begin{rem}
Let us stress that for the rest of this paper, the term 'diagram'
will refer exclusively to a directed and weighted graph in the sense
just described. Hence the word will never refer to an (affine) Dynkin
diagram nor will there be any possibility for confusion with the common
meaning of a diagram, which is that of an undirected graph. 
\end{rem}

\section{\label{sec:The--diagrams}The diagrams }

\noindent The following notion will be involved in the definition
of all diagrams for any type.

\begin{defn}
\label{def:nabla}By a \emph{$(\nabla,x)$ graph} we refer to the
infinite undirected graph$$ 
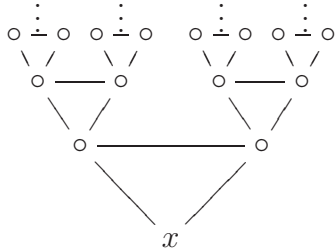
\begin{xy} 0;<0.25pt,0pt>:<0pt,-0.24pt>::  (237,351) *+{x} ="0", (375,202) *+{\circ} ="1", (100,202) *+{\circ} ="2", (436,101) *+{\circ} ="3", (310,101) *+{\circ} ="4", (162,101) *+{\circ} ="5", (36,101) *+{\circ} ="6", (475,27) *+{\circ} ="7", (400,27) *+{\circ} ="8", (350,27) *+{\circ} ="9", (275,27) *+{\circ} ="10", (200,27) *+{\circ} ="11", (125,27) *+{\circ} ="12", (75,27) *+{\circ} ="13", (0,27) *+{\circ} ="14", (36,-14) *+{\vdots} ="15", (162,-14) *+{\vdots} ="16", (314,-14) *+{\vdots} ="17", (441,-14) *+{\vdots} ="18", "1", {\ar@{-}"0"}, "0", {\ar@{-}"2"}, "2", {\ar@{-}"1"}, "3", {\ar@{-}"1"}, "1", {\ar@{-}"4"}, "5", {\ar@{-}"2"}, "2", {\ar@{-}"6"}, "4", {\ar@{-}"3"}, "7", {\ar@{-}"3"}, "3", {\ar@{-}"8"}, "9", {\ar@{-}"4"}, "4", {\ar@{-}"10"}, "6", {\ar@{-}"5"}, "11", {\ar@{-}"5"}, "5", {\ar@{-}"12"}, "13", {\ar@{-}"6"}, "6", {\ar@{-}"14"}, "8", {\ar@{-}"7"}, "10", {\ar@{-}"9"}, "12", {\ar@{-}"11"}, "14", {\ar@{-}"13"}, \end{xy} $$with
uniquely determined vertex $x$. It is obtained by starting with an
infinite binary tree and joining the two children of each vertex by
an edge. We also write $\nabla$ graph for short. 
\end{defn}
\noindent Before we give the diagrams for the particular (affine)
Dynkin types, note that the notation used in parts (4) and (6) of
Theorem \ref{Theorem (mutation classes)} is slightly imprecise, as
it suggests the existence of some families of diagrams which will
not be defined, cf. the summary in subsection \ref{subsection: summary}.

\subsection{The diagrams for type $\mathbb{A}$}

\begin{defn}
\label{def: A}$A$ is the infinite undirected graph obtained by the
following construction. Let $(\nabla_{1}$$,x_{1}$) and $(\nabla_{2},x_{2}$)
be two disjoint $\nabla$ graphs. Identify $x_{1}$ and $x_{2}$.$$ 
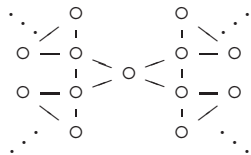
\begin{xy} 0;<.4pt,0pt>:<0pt,-.3pt>:: (150,50) *+{\circ} ="0", (100,75) *+{\circ} ="1", (150,100) *+{\circ} ="2", (50,50) *+{\circ} ="3", (50,100) *+{\circ} ="4", (50,0) *+{\circ} ="5", (0,50) *+{\circ} ="6", (150,0) *+{\circ} ="7", (200,50) *+{\circ} ="8", (200,100) *+{\circ} ="9", (150,150) *+{\circ} ="10", (0,100) *+{\circ} ="11", (50,150) *+{\circ} ="12", (0,150) *+{\iddots} ="13", (0,0) *+{\ddots} ="14", (200,0) *+{\iddots} ="15", (200,150) *+{\ddots} ="16", "1", {\ar@{-}"0"}, "0", {\ar@{-}"2"}, "0", {\ar@{-}"7"}, "8", {\ar@{-}"0"}, "2", {\ar@{-}"1"}, "3", {\ar@{-}"1"}, "1", {\ar@{-}"4"}, "2", {\ar@{-}"9"}, "10", {\ar@{-}"2"}, "4", {\ar@{-}"3"}, "5", {\ar@{-}"3"}, "3", {\ar@{-}"6"}, "11", {\ar@{-}"4"}, "4", {\ar@{-}"12"}, "6", {\ar@{-}"5"}, "7", {\ar@{-}"8"}, "9", {\ar@{-}"10"}, "12", {\ar@{-}"11"}, \end{xy} $$An
\textbf{\emph{$A$}} \emph{diagram} is a cyclically oriented simply-laced
diagram with underlying undirected graph equal to a connected and
full subgraph of $A$. 
\end{defn}

\subsection{The diagrams for type $\mathbb{B}$}

\begin{defn}
\label{def:B}Let $B$ be the infinite undirected weighted graph obtained
by the following construction. Let $(\nabla,x)$ be a $\nabla$ graph.
Assign weight two to both edges having $x$ as one endpoint and weight
one to all others.

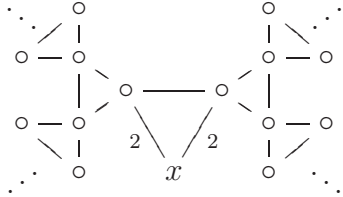
\[
\xyR{.01pc}\xyC{.4pc}\xymatrix{\ddots & \circ &  &  &  & \circ & \iddots\\
\circ & \circ\ar@{-}'[l]'[u]'[]'[dr]'[dd] &  &  &  & \circ\ar@{-}'[u]'[r]'[]'[dl]'[dd] & \circ\\
 &  & \circ\ar@{-}[rr] &  & \circ\ar@{-}[ddl]^{2}\\
\circ & \circ\ar@{-}'[l]'[d]'[]'[uu] &  &  &  & \circ\ar@{-}'[r]'[d]'[]'[uu] & \circ\\
\iddots & \circ &  & x\ar@{-}[uul]^{2} &  & \circ & \ddots}
\]
 A \textbf{\emph{$B$}} \emph{diagram} is a cyclically oriented diagram
with underlying undirected weighted graph equal to a connected and
full subgraph of $B$ containing the vertex $x$. 
\end{defn}

\subsection{The diagrams for type $\mathbb{D}$}

\begin{defn}
\label{def:D}1) For $n\in\mathbb{N}_{\geq3}$ let $D_{(\bigcirc,n)}$
denote the infinite undirected graph which is obtained by the following
construction. Let $\{a_{1},..,a_{n}\}$ be the set of vertices of
an $A_{n-1}^{(1)}$ graph labeled clockwise, $(\nabla_{i},x_{i})_{1\leq i\leq n}$
be $\nabla$ graphs. For all $1\leq i\leq n$ join $x_{i}$ to both
$a_{i}$ and $a_{i+1}$ by single edges, where $a_{n+1}$ denotes
the vertex $a_{1}$.$$ 
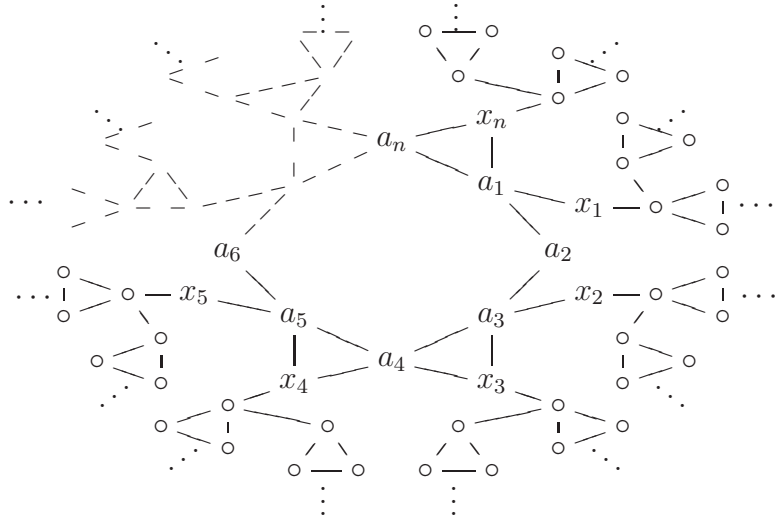
\begin{xy} 0;<0.9pt,0pt>:<0pt,-0.6pt>::  (154,84) *+{a_n} ="0", (196,112) *+{a_1} ="1", (224,154) *+{a_2} ="2", (196,195) *+{a_3} ="3", (154,223) *+{a_4} ="4", (113,195) *+{a_5} ="5", (85,154) *+{a_6} ="6", (113,112) *+{} ="7", (196,70) *+{x_n} ="8", (237,126) *+{x_1} ="9", (237,181) *+{x_2} ="10", (196,237) *+{x_3} ="11", (113,237) *+{x_4} ="12", (71,181) *+{x_5} ="13", (71,126) *+{} ="14", (113,70) *+{} ="15", (182,43) *+{\circ} ="16", (224,57) *+{\circ} ="17", (251,98) *+{\circ} ="18", (265,126) *+{\circ} ="19", (265,181) *+{\circ} ="20", (251,209) *+{\circ} ="21", (224,251) *+{\circ} ="22", (182,264) *+{\circ} ="23", (127,264) *+{\circ} ="24", (85,251) *+{\circ} ="25", (57,209) *+{\circ} ="26", (43,181) *+{\circ} ="27", (43,126) *+{} ="28", (57,98) *+{} ="29", (85,57) *+{} ="30", (127,43) *+{} ="31", (168,15) *+{\circ} ="32", (196,15) *+{\circ} ="33", (224,29) *+{\circ} ="34", (251,43) *+{\circ} ="35", (251,70) *+{\circ} ="36", (279,84) *+{\circ} ="37", (293,112) *+{\circ} ="38", (293,140) *+{\circ} ="39", (293,167) *+{\circ} ="40", (293,195) *+{\circ} ="41", (279,223) *+{\circ} ="42", (251,237) *+{\circ} ="43", (251,264) *+{\circ} ="44", (224,278) *+{\circ} ="45", (196,292) *+{\circ} ="46", (168,292) *+{\circ} ="47", (140,292) *+{\circ} ="48", (113,292) *+{\circ} ="49", (85,278) *+{\circ} ="50", (57,264) *+{\circ} ="51", (57,237) *+{\circ} ="52", (30,223) *+{\circ} ="53", (16,195) *+{\circ} ="54", (16,167) *+{\circ} ="55", (16,140) *+{} ="56", (16,112) *+{} ="57", (30,84) *+{} ="58", (57,70) *+{} ="59", (57,43) *+{} ="60", (85,29) *+{} ="61", (113,15) *+{} ="62", (140,15) *+{} ="63", (181,0) *+{\vdots} ="64", (243,22) *+{\iddots} ="65", (271,65) *+{\iddots} ="66", (307,125) *+{\cdots} ="67", (308,182) *+{\cdots} ="68", (271,239) *+{\ddots} ="69", (244,279) *+{\ddots} ="70", (180,305) *+{\vdots} ="71", (125,305) *+{\vdots} ="72", (66,279) *+{\iddots} ="73", (37,239) *+{\iddots} ="74", (3,183) *+{\cdots} ="75", (0,123) *+{\cdots} ="76", (35,65) *+{\ddots} ="77", (60,22) *+{\ddots} ="78", (125,1) *+{\vdots} ="79", "0", {\ar@{-}"1"}, "7", {\ar@{--}"0"}, "8", {\ar@{-}"0"}, "0", {\ar@{--}"15"}, "1", {\ar@{-}"2"}, "1", {\ar@{-}"8"}, "9", {\ar@{-}"1"}, "2", {\ar@{-}"3"}, "2", {\ar@{-}"9"}, "10", {\ar@{-}"2"}, "3", {\ar@{-}"4"}, "3", {\ar@{-}"10"}, "11", {\ar@{-}"3"}, "4", {\ar@{-}"5"}, "4", {\ar@{-}"11"}, "12", {\ar@{-}"4"}, "5", {\ar@{-}"6"}, "5", {\ar@{-}"12"}, "13", {\ar@{-}"5"}, "6", {\ar@{--}"7"}, "6", {\ar@{-}"13"}, "14", {\ar@{--}"6"}, "7", {\ar@{--}"14"}, "15", {\ar@{--}"7"}, "16", {\ar@{-}"8"}, "8", {\ar@{-}"17"}, "18", {\ar@{-}"9"}, "9", {\ar@{-}"19"}, "20", {\ar@{-}"10"}, "10", {\ar@{-}"21"}, "22", {\ar@{-}"11"}, "11", {\ar@{-}"23"}, "24", {\ar@{-}"12"}, "12", {\ar@{-}"25"}, "26", {\ar@{-}"13"}, "13", {\ar@{-}"27"}, "28", {\ar@{--}"14"}, "14", {\ar@{--}"29"}, "30", {\ar@{--}"15"}, "15", {\ar@{--}"31"}, "17", {\ar@{-}"16"}, "32", {\ar@{-}"16"}, "16", {\ar@{-}"33"}, "34", {\ar@{-}"17"}, "17", {\ar@{-}"35"}, "19", {\ar@{-}"18"}, "36", {\ar@{-}"18"}, "18", {\ar@{-}"37"}, "38", {\ar@{-}"19"}, "19", {\ar@{-}"39"}, "21", {\ar@{-}"20"}, "40", {\ar@{-}"20"}, "20", {\ar@{-}"41"}, "42", {\ar@{-}"21"}, "21", {\ar@{-}"43"}, "23", {\ar@{-}"22"}, "44", {\ar@{-}"22"}, "22", {\ar@{-}"45"}, "46", {\ar@{-}"23"}, "23", {\ar@{-}"47"}, "25", {\ar@{-}"24"}, "48", {\ar@{-}"24"}, "24", {\ar@{-}"49"}, "50", {\ar@{-}"25"}, "25", {\ar@{-}"51"}, "27", {\ar@{-}"26"}, "52", {\ar@{-}"26"}, "26", {\ar@{-}"53"}, "54", {\ar@{-}"27"}, "27", {\ar@{-}"55"}, "29", {\ar@{--}"28"}, "56", {\ar@{--}"28"}, "28", {\ar@{--}"57"}, "58", {\ar@{--}"29"}, "29", {\ar@{--}"59"}, "31", {\ar@{--}"30"}, "60", {\ar@{--}"30"}, "30", {\ar@{--}"61"}, "62", {\ar@{--}"31"}, "31", {\ar@{--}"63"}, "33", {\ar@{-}"32"}, "35", {\ar@{-}"34"}, "37", {\ar@{-}"36"}, "39", {\ar@{-}"38"}, "41", {\ar@{-}"40"}, "43", {\ar@{-}"42"}, "45", {\ar@{-}"44"}, "47", {\ar@{-}"46"}, "49", {\ar@{-}"48"}, "51", {\ar@{-}"50"}, "53", {\ar@{-}"52"}, "55", {\ar@{-}"54"}, "57", {\ar@{--}"56"}, "59", {\ar@{--}"58"}, "61", {\ar@{--}"60"}, "63", {\ar@{--}"62"}, \end{xy} $$The
subgraph of $D_{(\bigcirc,n)}$ induced on $\{a_{i}\}_{1\leq i\leq n}$
be denoted by $(\bigcirc,n)$. A \emph{$D_{(\bigcirc,n)}$ diagram}
$\Gamma$ is a cyclically oriented simply-laced diagram with underlying
undirected graph equal to a connected full subgraph of $D_{(\bigcirc,n)}$
containing the vertices of $(\bigcirc,n)$ such that $|\Gamma_{0}|\geq5$.

\noindent 2) $D_{\square}$ is the infinite undirected graph obtained
by the following construction. Let $a_{1},a_{2}$ be two single vertices
and $(\nabla_{i},x_{i})_{1\leq i\leq2}$ be two $\nabla$ graphs.
Join each of $x_{1},x_{2}$ to each of $a_{1},a_{2}$ by a single
edge.

By $\square$ we denote the subgraph of $D_{\square}$ induced on
$\{x_{1},a_{1},x_{2},a_{2}\}$. A \emph{$D_{\square}$ diagram} is
a cyclically oriented simply-laced diagram with underlying undirected
graph equal to a connected full subgraph of $D_{\square}$ containing
the vertices of $\square$.

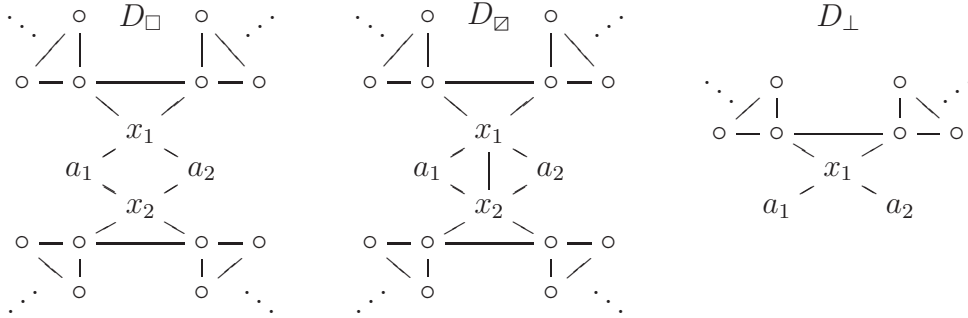
\[
\xyR{.01pc}\xyC{.2pc}\xymatrix{\ddots & \circ & D_{\square} & \circ & \iddots &  &  & \ddots & \circ & D_{\boxslash} & \circ & \iddots &  &  &  &  & D_{\bot}\\
\circ & \circ\ar@{-}'[l]'[u]'[]'[dr]'[rr]'[] &  & \circ\ar@{-}'[r]'[u]'[] & \circ &  &  & \circ & \circ\ar@{-}'[l]'[u]'[]'[rr] &  & \circ\ar@{-}'[r]'[u]'[]'[dl]'[ll] & \circ &  &  & \ddots & \circ &  & \circ & \iddots\\
 &  & x_{1}\ar@{-}'[dl]'[dd]'[dr]'[] &  &  &  &  &  &  & x_{1}\ar@{-}'[dl]'[dd]'[dr]'[]'[dd] &  &  &  &  & \circ & \circ\ar@{-}'[l]'[u]'[]'[rr] &  & \circ\ar@{-}'[r]'[u]'[]'[dl]'[ll] & \circ\\
 & a_{1} &  & a_{2} &  &  &  &  & a_{1} &  & a_{2} &  &  &  &  &  & x_{1}\\
 &  & x_{2} &  &  &  &  &  &  & x_{2} &  &  &  &  &  & a_{1}\ar@{-}'[ur]'[rr] &  & a_{2}\\
\circ & \circ\ar@{-}'[l]'[d]'[]'[ur]'[rr] &  & \circ\ar@{-}'[r]'[d]'[]'[ll] & \circ &  &  & \circ & \circ\ar@{-}'[l]'[d]'[]'[rr] &  & \circ\ar@{-}'[r]'[d]'[]'[ul]'[ll] & \circ\\
\iddots & \circ &  & \circ & \ddots &  &  & \iddots & \circ &  & \circ & \ddots}
\]

\noindent 3) Let $D_{\boxslash}$ denote the infinite undirected
graph which is obtained by inserting one edge between $x_{1}$ and
$x_{2}$ in $D_{\square}$.

The subgraph of $D_{\boxslash}$ induced by $\{a_{i},x_{i}\}_{1\leq i\leq2}$
is denoted by $\boxslash$. A $D_{\boxslash}$ \emph{diagram} is a
cyclically oriented simply-laced diagram with underlying undirected
graph equal to a connected full subgraph of $D_{\boxslash}$ containing
the vertices of $\boxslash$.\\

\noindent 4) $D_{\bot}$ is the infinite undirected graph obtained
by deleting the $\nabla$ graph attached to $x_{2}$ in $D_{\square}$
including the vertex $x_{2}$.

By $\bot$ we denote the subgraph of $D_{\bot}$ induced on $\{a_{1},a_{2},x_{1}\}$.
A $D_{\bot}$ \emph{diagram} is a cyclically oriented simply-laced
diagram with underlying undirected graph equal to a connected and
full subgraph of $D_{\bot}$ containing the vertices of $\bot$. 
\end{defn}
\begin{rem}
Note that the families of $D_{\star}$ diagrams, where $\star\in\bigl\{(\bigcirc,n)_{n\geq3},\square,\boxslash,\bot\bigr\}$,
are mutually disjoint. Indeed, the crucial point is that: 
\end{rem}
\begin{itemize}
\item a $D_{(\bigcirc,4)}$ diagram can always be distinguished from a $D_{\square}$
diagram. This is because the first one is required to obtain at least
five vertices, so it always contains a subgraph of the form \emph{i)}
as shown below. But a $D_{\square}$ diagram never does. 
\item a $D_{(\bigcirc,3)}$ diagram $\Gamma$ can always be distinguished
from a $D_{\boxslash}$ diagram even if two of $\{x_{i}\}_{1\leq i\leq3}$
are absent (cf. \emph{ii)} for the general case). This is because
the first one always contains a subgraph of the form \emph{iii)} since
$|\Gamma_{0}|\geq5$, while the second one never does. 
\end{itemize}
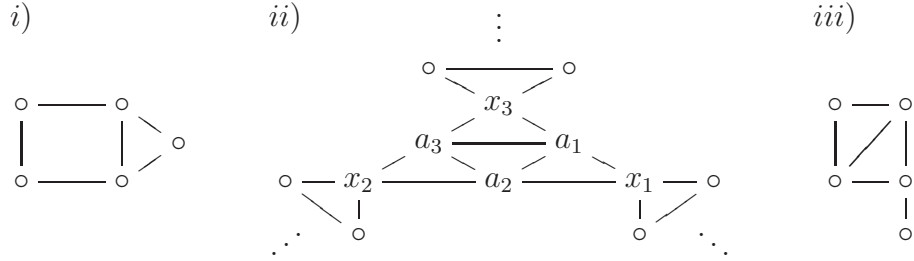
\[
\xyR{.05pc}\xyC{.7pc}\xymatrix{i) &  &  &  &  & ii) &  &  & \vdots &  &  &  &  & iii)\\
 &  &  &  &  &  &  & \circ &  & \circ\\
\circ\ar@{-}'[rr]'[ddrr]'[dd]'[] &  & \circ &  &  &  &  &  & x_{3}\ar@{-}'[ul]'[ur]'[]\ar@{-}'[dl]'[dr]'[] &  &  &  &  & \circ\ar@{-}'[r]'[dd]'[] & \circ\\
 &  &  & \circ\ar@{-}'[dl]'[ul]'[] &  &  &  & a_{3}\ar@{-}'[dl]'[dr]'[] &  & a_{1}\ar@{-}'[dl]'[dr]'[]\\
\circ &  & \circ &  &  & \circ & x_{2}\ar@{-}'[l]'[d]'[] &  & a_{2} &  & x_{1}\ar@{-}'[d]'[r]'[] & \circ &  & \circ & \circ\ar@{-}'[l]'[uu]'[]'[d]\\
 &  &  &  &  & \iddots & \circ &  &  &  & \circ & \ddots &  &  & \circ}
\]

\subsection{The diagrams for type $\mathbb{B}^{(1)}$\label{sub:The-diagrams-for B^(1)}}

Loosely speaking, what should a diagram of mutation type $\mathbb{B}^{(1)}$
look like? We would expect it to be a diagram obtained by somehow
gluing a diagram of mutation type $\mathbb{B}$ with a diagram of
mutation type $\mathbb{D}$. Indeed, this is partially true. For example,
we will encounter all diagrams induced by connected subgraphs of either
of the following infinite graphs containing both 'characteristic subgraphs'
of $B$ and $D_{\boxslash}$\[
\xyR{.01pc}\xyC{.4pc}
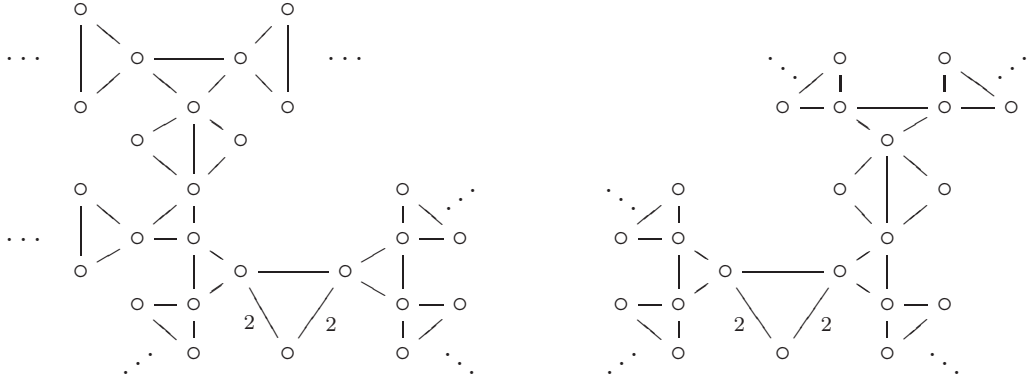
\xymatrix{ & \circ &  &  &  & \circ\\
\dots &  & \circ\ar@{-}'[dl]'[ul]'[] &  & \circ\ar@{-}'[ur]'[dr]'[] &  & \dots &  &  &  &  &  &  &  &  & \ddots & \circ\ar@{-}'[d]'[dl]'[] &  & \circ\ar@{-}'[dr]'[d]'[] & \iddots\\
 & \circ &  & \circ\ar@{-}'[ul]'[ur]'[] &  & \circ &  &  &  &  &  &  &  &  &  & \circ & \circ &  & \circ & \circ\\
 &  & \circ &  & \circ &  &  &  &  &  &  &  &  &  &  &  &  & \circ\ar@{-}'[ul]'[ur]'[]\\
 & \circ &  & \circ\ar@{-}'[ul]'[uu]'[]'[ur]'[uu] &  &  &  & \circ & \iddots &  &  &  & \ddots & \circ &  &  & \circ &  & \circ\\
\dots &  & \circ\ar@{-}'[ul]'[dl]'[] & \circ\ar@{-}'[l]'[u]'[]'[dr]'[dd] &  &  &  & \circ\ar@{-}'[u]'[r]'[]'[dl]'[dd] & \circ &  &  &  & \circ & \circ\ar@{-}'[l]'[u]'[]'[dr]'[dd] &  &  &  & \circ\ar@{-}'[ur]'[uu]'[]'[ul]'[uu]\\
 & \circ &  &  & \circ\ar@{-}[rr] &  & \circ\ar@{-}[ddl]^{2} &  &  &  &  &  &  &  & \circ\ar@{-}[rr] &  & \circ\ar@{-}[ddl]^{2}\\
 &  & \circ & \circ\ar@{-}'[l]'[d]'[]'[uu] &  &  &  & \circ\ar@{-}'[r]'[d]'[]'[uu] & \circ &  &  &  & \circ & \circ\ar@{-}'[l]'[d]'[]'[uu] &  &  &  & \circ\ar@{-}'[r]'[d]'[]'[uu]'[ul]'[] & \circ\\
 &  & \iddots & \circ &  & \circ\ar@{-}[uul]^{2} &  & \circ & \ddots &  &  &  & \iddots & \circ &  & \circ\ar@{-}[uul]^{2} &  & \circ & \ddots}
\]
 However, the description of these families and their behaviour under
mutation poses a problem concerning notation. As the reader might
guess, we can 'glue' the infinite graphs $B$ and $D_{\star}$, where
$\star\in\{(\bigcirc,n)_{n\geq3},\square,\boxslash,\bot\}$, at any
point of a $\nabla$ graph. But if we were to continue with our habit
of labeling diagrams by the infinite graphs they correspond to and
additionally keep track of the point of the gluing, this would produce
far more labels than we care for. After all, the subdiagrams we obtain
from say the above two graphs both locally look like a $B$ or $D_{\boxslash}$
diagram, with some connection via the $\nabla$ part. Thus, their
behaviour under mutation 'mostly coincides', even though they were
obtained by different ways of gluing.

Hence we decide to change our strategy concerning notation: we allow
for a small amount of ambiguity in the construction of diagrams (i.e.
we ignore the gluing point) in order to concentrate on the defining
characteristic of the families of diagrams we are interested in. This
is made precise by

\begin{defn}
\label{def: B_star_B}For any $\star\in\{(\bigcirc,n)_{n\geq3},\square,\boxslash,\bot\}$
a $B_{\star,B}$ \emph{diagram} is a diagram obtained by the following
construction. Let $\Gamma$ be a $D_{\star}$ diagram, $\Gamma'$
be a $B$ diagram and fix the notations of definitions \ref{def:D}
and \ref{def:B}, where we add primes to the labels used in \ref{def:B}.
Choose a vertex $y$ of $\Gamma$ which is not labeled $a_{i}$ and
$y'\in\Gamma_{0}'\setminus\{x'\}$ such that the following is true: 
\end{defn}
\begin{itemize}
\item if $y\in\star_{0}$ then all neighbours of $y$ are contained in $\star_{0}$. 
\item $y'$ is of degree at most two, as is $y$ if $y\notin\star_{0}$. 
\end{itemize}
Then identify $y$ and $y'$.

Given a $B_{\star,B}$ diagram $\bar{\Gamma}$, let $\omega_{\bar{\Gamma}}$
be the shortest of all paths in the underlying undirected graph starting
in any of $x_{i}\in\Gamma_{0}$ and ending in $x'\in\Gamma'_{0}$
(again cf. the notation of definitions \ref{def: A} and \ref{def:B}).
Then the width $w(\bar{\Gamma})$ of $\bar{\Gamma}$ is defined as
the length of $\omega_{\bar{\Gamma}}$ minus one.

\begin{example}
We give two examples, the first showing a $B_{(\bigcirc,4),B}$ diagram
of width two, the second a $B_{\boxslash,B}$ of width zero. In the
first case, $\star=(\bigcirc,4)$ so $\star_{0}=\{a_{1},a_{2},a_{3},a_{4}\}\cup\{x_{1},x_{2},x_{3},x_{4}\}$,
while in the second $\star=\boxslash$, hence $\star_{0}=\{a_{1},a_{2}\}\cup\{x_{1},x_{2}\}$.
Note that, as exemplified in i), $\star_{0}$ is not necessarily contained
in the vertex set of every $B_{\star,B}$ diagram. \[
\xyR{.7pc}\xyC{.7pc}
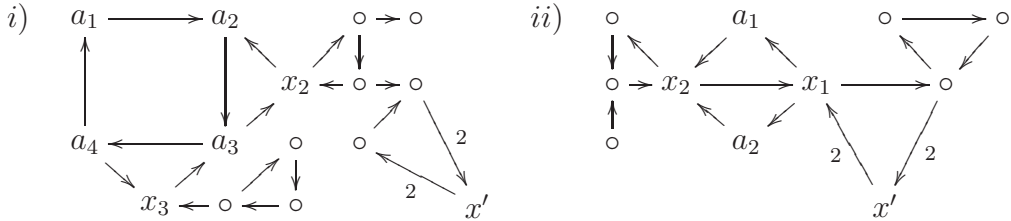
\xymatrix{i) & a_{1}\ar[rr] &  & a_{2}\ar[dd] &  & \circ\ar[d]\ar[r] & \circ &  & ii) & \circ\ar[d] &  & a_{1}\ar[dl] &  & \circ\ar[rr] &  & \circ\ar[dl]\\
 &  &  &  & x_{2}\ar[ur]\ar[ul] & \circ\ar[l]\ar[r] & \circ\ar[ddr]^{2} &  &  & \circ\ar[r] & x_{2}\ar[ul]\ar[rr] &  & x_{1}\ar[ul]\ar[dl]\ar[rr] &  & \circ\ar[ddl]^{2}\ar[ul]\\
 & a_{4}\ar[uu]\ar[dr] &  & a_{3}\ar[ll]\ar[ur] & \circ\ar[d] & \circ\ar[ur] &  &  &  & \circ\ar[u] &  & a_{2}\ar[ul]\\
 &  & x_{3}\ar[ur] & \circ\ar[l]\ar[ur] & \circ\ar[l] &  &  & x'\ar[ull]^{2} &  &  &  &  &  & x'\ar[uul]^{2}}
\]

\end{example}
\begin{rem}
\label{rem:only class of B_star_B}Note that for any $\star\in\{(\bigcirc,n)_{n\geq3},\square,\boxslash,\bot\}$
there are usually many different pairs of $D_{\star}$ and $B$ diagrams
which yield the same $B_{\star,B}$ diagram (for certain choices of
vertices $y$ and $y'$ as described in the construction of definition
\ref{def: B_star_B}). Cf. \emph{i)} in the example above, where 'splitting'
at any vertex in the horizontal line of $x_{2}$ yields such a pair.
Let us stress that - as discussed in the beginning of this section
- we are not interested in the combinatorics of such 'splittings',
but only in the classes of $B_{\star,B}$ diagrams, where $\star\in\{(\bigcirc,n)_{n\geq3},\square,\boxslash,\bot\}$,
themselves (cf. Remarks \ref{rem:only class of C_B_B} and \ref{rem:only class of D_star_star'}). 
\end{rem}
Let us assume for a moment that we had already worked through Facts
\ref{Lemma: B mutations steps} and \ref{Lemma: D mutation steps}.
Mutation being a local concept, we would then expect that mutating
any member $\Gamma$ of the family of diagrams just defined yields
just another member of this family in 'most cases'. Indeed, this will
turn out to be true for $\omega(\Gamma)>1$ (Fact \ref{lem: B^(1) mutation steps}
1) a)), but also in the case where $\omega(\Gamma)=0$ and mutation
is performed at any point except the starting point of $\omega_{\Gamma}$
in the '$D_{\star}$ diagram part' of $\Gamma$ (cf. Fact \ref{lem: B^(1) mutation steps}
1) b)). Clearly, the question arises what happens if we mutate in
case $\omega(\Gamma)=0$ at the vertex just mentioned. As it turns
out, we might encounter members of some new families of diagrams (cf.
Fact \ref{lem: B^(1) mutation steps} 1) c)), which are:

\begin{defn}
\label{def: B_star_wedge_B} 1) Let $B_{\star\wedge B}$, where $\star\in\{\square,\boxslash\}$,
be the infinite undirected weighted graph obtained by the following
construction. Delete the $\nabla$ graph attached to $x_{1}$ in $D_{\star}$
without removing $x_{1}$. Assign weight two to all edges with endpoint
$x_{1}$ and weight one to all others.

By $\star\wedge B$ we denote the subgraph of $B_{\star\wedge B}$
induced on $\{x_{i},a_{i}\}_{1\leq i\leq2}$. A \textbf{\emph{$B_{\star\wedge B}$}}
\emph{diagram} is a cyclically oriented diagram with underlying undirected
weighted graph equal to a connected full subgraph of $B_{\star\wedge B}$
containing the vertices of $\star\wedge B$.

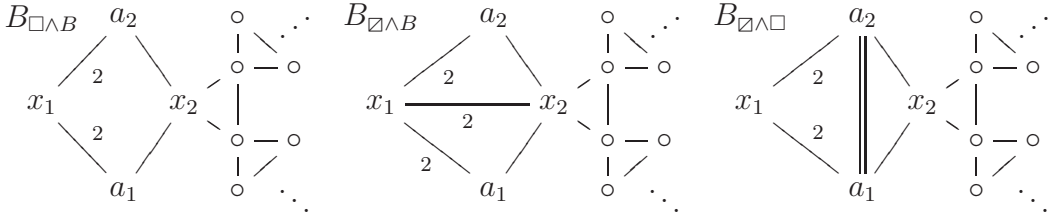
\[
\xyR{.05pc}\xyC{.38pc}\xymatrix{B_{\square\wedge B} & a_{2}\ar@{-}[ddr]^{}\ar@{-}[ddl]_{}^{2} &  & \circ & \iddots & B_{\boxslash\wedge B} &  & a_{2}\ar@{-}[ddr]^{}\ar@{-}[ddll]_{}^{2} &  & \circ & \iddots & B_{\boxslash\wedge\square} &  & a_{2}\ar@{=}[dddd]\ar@{-}[ddll]_{}^{2} &  & \circ & \iddots\\
 &  &  & \circ\ar@{-}'[u]'[r]'[] & \circ &  &  &  &  & \circ\ar@{-}'[u]'[r]'[] & \circ &  &  &  &  & \circ\ar@{-}'[u]'[r]'[] & \circ\\
x_{1} &  & x_{2}\ar@{-}'[ur]'[dr]'[] &  &  & x_{1}\ar@{-}[rrr]_{2}^{} &  &  & x_{2}\ar@{-}'[ur]'[dr]'[] &  &  & x_{1} &  &  & x_{2}\ar@{-}'[ur]'[dr]'[]\\
 &  &  & \circ\ar@{-}'[d]'[r]'[] & \circ &  &  &  &  & \circ\ar@{-}'[d]'[r]'[] & \circ &  &  &  &  & \circ\ar@{-}'[d]'[r]'[] & \circ\\
 & a_{1}\ar@{-}[uur]\ar@{-}[uul]_{2}^{} &  & \circ & \ddots &  &  & a_{1}\ar@{-}[uur]\ar@{-}[uull]^{2} &  & \circ & \ddots &  &  & a_{1}\ar@{-}'[uur]'[uuuu]\ar@{-}[uull]_{2}^{} &  & \circ & \ddots}
\]
2) $B_{\boxslash\wedge\square}$ is the infinite undirected graph
obtained by joining $a_{1}$ and $a_{2}$ in $B_{\square\wedge B}$
by two edges of weight one.

The subgraph of $B_{\boxslash\wedge\square}$ induced on $\{x_{i},a_{i}\}_{1\leq i\leq2}$
is denoted $\boxslash\wedge\square$. A \textbf{\emph{$B_{\boxslash\wedge\square}$}}
\emph{diagram} is a cyclically oriented diagram with underlying undirected
weighted graph equal to a connected full subgraph of $B_{\boxslash\wedge\square}$
containing the vertices of $\boxslash\wedge\square$.\\

\noindent 3) Let $B_{(\bigcirc,n)\wedge B}$ be the infinite undirected
weighted graph obtained by the following construction. Delete the
$\nabla$ graph attached to $x_{1}$ in $D_{(\bigcirc,n)}$ without
removing the vertex itself. Assign weight two to both edges with endpoint
$x_{1}$ and weight one to all others.

For $\star\in\{B,\overleftrightarrow{B}\}$ the subgraph of $B_{(\bigcirc,n)\wedge B}$
induced on $\{a_{i}\}_{1\leq i\leq n}\cup\{x_{1}\}$ is denoted by
$(\bigcirc,n)\wedge\star$ (cf. Remark \ref{rem: B vs overarrow B business}).
A $B_{(\bigcirc,n)\wedge\star}$ \emph{diagram} is a diagram $\Gamma$
with underlying undirected weighted graph equal to a connected full
subgraph of $B_{(\bigcirc,n)\wedge B}$ containing the vertices of
$(\bigcirc,n)\wedge\star$ such that the following conditions are
satisfied: 
\end{defn}
\begin{itemize}
\item if $\star=B$, then $\Gamma$ is cyclically oriented. 
\item if $\star=\overleftrightarrow{B}$, let $e$ denote the edge connecting
$a_{1}$ and $a_{2}$. Then the graph $(\Gamma_{0},\Gamma_{1}\setminus e)$
is cyclically oriented, while the subgraph induced on $\{a_{i}\}_{1\leq i\leq n}$
is not. 
\end{itemize}
\begin{rem}
\label{rem: B vs overarrow B business} The reader should not get
confused by the fact that the symbol $\overleftrightarrow{B}$ itself
is not defined. In particular it does not refer to any particular
subgraph of some infinite graph as do say $\square$ and $\boxslash$
. The notation simply serves as a mnemonic device. All that we care
for is the subgraph $(\bigcirc,n)\wedge\overleftrightarrow{B}$, which
is actually the same as $(\bigcirc,n)\wedge B$, since both do not
carry any orientation. It is the family of diagrams associated to
them after a particular choice of orientation that justifies the choice
of notation. 
\end{rem}
\begin{example}
The first example is a $B_{(\bigcirc,3)\wedge B}$, the second a $B_{(\bigcirc,3)\wedge\overleftrightarrow{B}}$
diagram. We indicate possible enlargements by dots.\[
\xyR{.5pc}\xyC{.5pc}
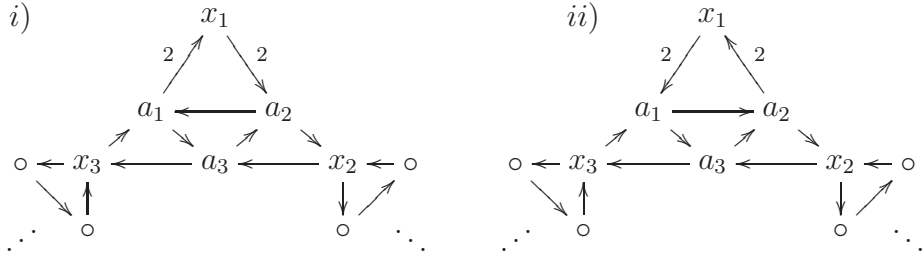
\xymatrix{i) &  &  & x_{1}\ar[ddr]^{2} &  &  &  &  &  & ii) &  & x_{1}\ar[ddl]_{2}\\
\\ &  & a_{1}\ar[dr]\ar[uur]^{2} &  & a_{2}\ar[ll]\ar[dr] &  &  &  &  &  & a_{1}\ar[dr]\ar[rr] &  & a_{2}\ar[uul]_{2}^{}\ar[dr]\\
\circ\ar[dr] & x_{3}\ar[l]\ar[ur] &  & a_{3}\ar[ur]\ar[ll] &  & x_{2}\ar[ll]\ar[d] & \circ\ar[l] &  & \circ\ar[dr] & x_{3}\ar[l]\ar[ur] &  & a_{3}\ar[ur]\ar[ll] &  & x_{2}\ar[ll]\ar[d] & \circ\ar[l]\\
\iddots & \circ\ar[u] &  &  &  & \circ\ar[ur] & \ddots &  & \iddots & \circ\ar[u] &  &  &  & \circ\ar[ur] & \ddots}
\]

\end{example}

\subsection{The diagrams for type $\mathbb{C}^{(1)}$}

Analogous to the case of type $\mathbb{B}^{(1)}$, 'most' diagrams
of mutation type $\mathbb{C}^{(1)}$ are obtained by 'gluing' diagrams
we already encountered, this time being of mutation type $\mathbb{C}$.
Again, there are many more ways to do this than we actually care for.
Hence we need to adjust our definitions to compensate for the combinatorial
likeness of the families of diagrams involved.

\begin{defn}
\label{def: C_B_B}A $C_{B,B}$ \emph{diagram} is a diagram obtained
by the following construction. Let $\Gamma,\Gamma'$ be two $B$ diagrams
and fix the notations of definition \ref{def:B} (where we add primes
to the labels used in \ref{def:B} for the vertices of $\Gamma'$).
Choose $y\in\Gamma_{0}\setminus\{x\}$ and $y'\in\Gamma_{0}'\setminus\{x'\}$
such that both vertices are of degree at most two. Identify $y$ and
$y'$.

For a $C_{B,B}$ diagram $\bar{\Gamma}$, let $\omega_{\bar{\Gamma}}$
be the shortest of all paths in the underlying undirected graph from
$x$ to $x'$. Then the width $w(\bar{\Gamma})$ of $\bar{\Gamma}$
is defined as the length of $\omega_{\bar{\Gamma}}$ minus two. 
\end{defn}
\begin{example}
The first example shows a $C_{B,B}$ diagram of width zero, the second
a $C_{B,B}$ diagram of width one.\[
\xyR{.05pc}\xyC{1pc}
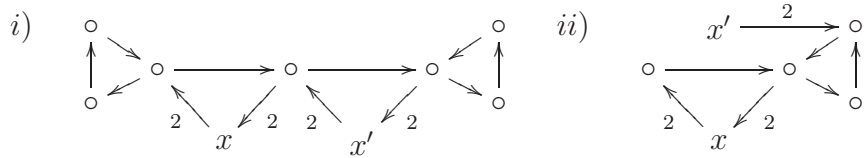
\xymatrix{i) & \circ\ar[dr] &  &  &  &  &  & \circ\ar[dl] & ii) &  & x'\ar[rr]^{2} &  & \circ\ar[dl]\\
 &  & \circ\ar[dl]\ar[rr] &  & \circ\ar[rr]\ar[ddl]^{2} &  & \circ\ar[ddl]^{2}\ar[dr] &  &  & \circ\ar[rr] &  & \circ\ar[ddl]^{2}\ar[dr]\\
 & \circ\ar[uu] &  &  &  &  &  & \circ\ar[uu] &  &  &  &  & \circ\ar[uu]\\
 &  &  & x\ar[uul]^{2} &  & x'\ar[uul]^{2} &  &  &  &  & x\ar[uul]^{2}}
\]

\end{example}
\begin{rem}
\label{rem:only class of C_B_B}There are usually many different pairs
of two $B$ diagrams which yield the same $C_{B,B}$ diagram. Cf.
\emph{ii)} in the example above, where 'splitting' at any vertex in
the right-hand triangle except at the lowest one yields such a pair.
We are not interested in the combinatorics of such 'splittings', but
only in the class of $C_{B,B}$ diagrams itself (cf. Remarks \ref{rem:only class of B_star_B}
and \ref{rem:only class of D_star_star'}). 
\end{rem}
\noindent Again, mutating these diagrams will yield other members
of the family just defined, at least for width greater than one (Fact
\ref{lem: C^(1) mutation steps} 1) a)). New diagrams show up if we
do certain mutations for the case where the width equals zero (Fact
\ref{lem: C^(1) mutation steps} 1) b)). These are:

\begin{defn}
\label{def: C_B_wedge_B} Let $C_{B\wedge B}$ be the infinite undirected
weighted graph obtained by the following construction. Let $(\nabla,\bullet)$
be a $\nabla$ graph, $x_{1},x_{2}$ two single vertices. Join each
of $x_{1},x_{2}$ with $\bullet$ by an edge of weight two, join both
of them by an edge of weight four and assign weight one to all other
edges.

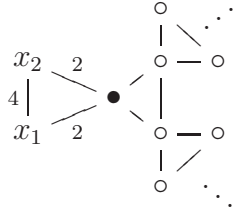
\[
\xyR{.05pc}\xyC{.4pc}\xymatrix{ &  &  & \circ & \iddots\\
x_{2}\ar@{-}[drr]^{2} &  &  & \circ\ar@{-}'[u]'[r]'[] & \circ\\
 &  & \bullet\ar@{-}[dll]^{2}\ar@{-}'[ur]'[dr]'[]\\
x_{1}\ar@{-}[uu]^{4} &  &  & \circ\ar@{-}'[d]'[r]'[] & \circ\\
 &  &  & \circ & \ddots}
\]
 The subgraph of $C_{B\wedge B}$ induced on $\{x_{1},x_{2},\bullet\}$
is denoted $B\wedge B$. A \textbf{\emph{$C_{B\wedge B}$}} \emph{diagram}
is a cyclically oriented diagram with underlying undirected weighted
graph equal to a connected full subgraph of $B$ containing the vertices
of $B\wedge B$. 
\end{defn}

\subsection{The diagrams for type $\mathbb{D}^{(1)}$\label{sub:The--diagrams for D^(1)}}

We assume that the reader is familiar with the discussion of type
$\mathbb{B}^{(1)}$ in order not to be puzzled by the following

\begin{defn}
\label{def: D_star_star'}A $D_{\star,\star'}$ \emph{diagram}, where
$\star,\star'\in\{(\bigcirc,n)_{n\geq3},\square,\boxslash,\bot\}$,
is a diagram obtained by the following construction. Let $\Gamma$
be $D_{\star}$ diagram, $\Gamma'$ be a $D_{\star'}$ diagram and
fix the notations of definition \ref{def:D} (where we add primes
to the labels used in \ref{def:D} for the vertices of $\Gamma'$).
Choose a vertex $y$ of $\Gamma$ which is not labeled $a_{i}$ and
a vertex $y'$ of $\Gamma'$ which is not labeled $a'_{j}$ such that
the following conditions are fulfilled: 
\end{defn}
\begin{itemize}
\item if $y\in\star_{0}$ then all neighbours of $y$ belong to $\star_{0}$
and $y'\notin\star_{0}'$. 
\item if $y'\in\star'_{0}$ then all neighbours of $y'$ belong to $\star'_{0}$
and $y\notin\star_{0}$. 
\item any of $y,y'$ not contained in $\star_{0}$ respectively $\star_{0}'$
is of degree at most two. 
\end{itemize}
Finally, identify $y$ and $y'$.

Given a $D_{\star,\star'}$ diagram $\bar{\Gamma}$, let $\omega_{\bar{\Gamma}}$
be the shortest of all paths in the underlying undirected graph from
any of $x_{i}\in\Gamma_{0}$ to any of $x'_{j}\in\Gamma'_{0}$. Then
the width $w(\bar{\Gamma})$ of $\bar{\Gamma}$ is defined as follows:

\begin{itemize}
\item if $(\star,\star')=((\bigcirc,n),\bigcirc,m))$ then $w(\bar{\Gamma})$
is equal to the length of $\omega_{\bar{\Gamma}}$. 
\item otherwise $w(\bar{\Gamma})$ is equal to the length of $\omega_{\bar{\Gamma}}$
minus one. 
\end{itemize}
\begin{example}
We give three examples. The first is a $D_{(\bigcirc,4),\bot}$ diagram
of width two, the second is a $D_{\boxslash,\boxslash}$ diagram of
width zero, the last is a $D_{(\bigcirc,3),(\bigcirc,3)}$ diagram
of width zero. \[
\xyR{.7pc}\xyC{.7pc}
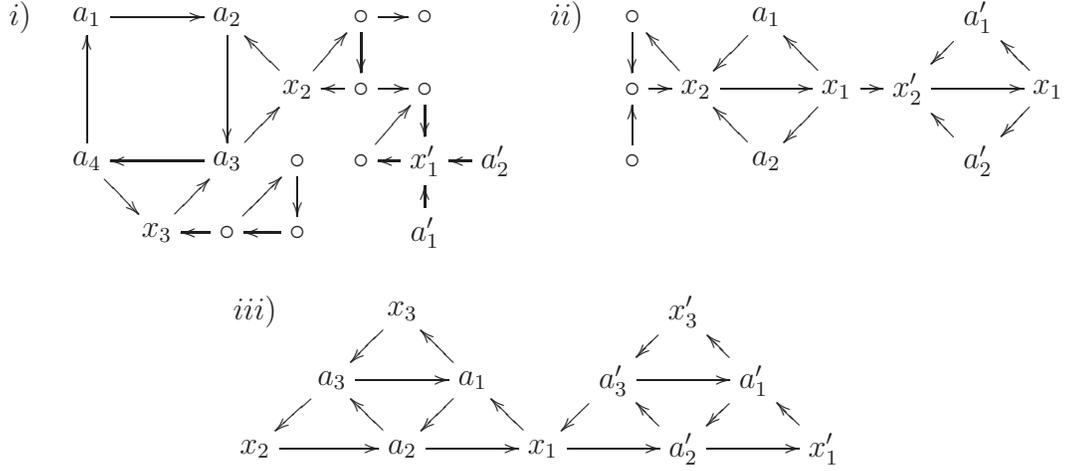
\xymatrix{i) & a_{1}\ar[rr] &  & a_{2}\ar[dd] &  & \circ\ar[d]\ar[r] & \circ &  & ii) & \circ\ar[d] &  & a_{1}\ar[dl] &  &  & a'_{1}\ar[dl]\\
 &  &  &  & x_{2}\ar[ur]\ar[ul] & \circ\ar[l]\ar[r] & \circ\ar[d] &  &  & \circ\ar[r] & x_{2}\ar[ul]\ar[rr] &  & x_{1}\ar[ul]\ar[dl]\ar[r] & x'_{2}\ar[rr] &  & x_{1}\ar[dl]\ar[ul]\\
 & a_{4}\ar[uu]\ar[dr] &  & a_{3}\ar[ll]\ar[ur] & \circ\ar[d] & \circ\ar[ur] & x_{1}'\ar[l] & a'_{2}\ar[l] &  & \circ\ar[u] &  & a_{2}\ar[ul] &  &  & a'_{2}\ar[ul]\\
 &  & x_{3}\ar[ur] & \circ\ar[l]\ar[ur] & \circ\ar[l] &  & a'_{1}\ar[u]}
\]
 \[
\xyR{.6pc}\xyC{.7pc}\xymatrix{iii) &  & x_{3}\ar[dl] &  &  &  & x'_{3}\ar[dl]\\
 & a_{3}\ar[rr]\ar[dl] &  & a_{1}\ar[dl]\ar[ul] &  & a'_{3}\ar[rr]\ar[dl] &  & a'_{1}\ar[dl]\ar[ul]\\
x_{2}\ar[rr] &  & a_{2}\ar[ul]\ar[rr] &  & x_{1}\ar[rr]\ar[ul] &  & a'_{2}\ar[ul]\ar[rr] &  & x'_{1}\ar[ul]}
\]

\end{example}
\begin{rem}
\label{rem:only class of D_star_star'}Note that there are usually
many different pairs of $D_{\star}$ and $D_{\star'}$ diagrams, where
$\star,\star'\in\{(\bigcirc,n)_{n\geq3},\square,\boxslash,\bot\}$,
which yield the same $D_{\star,\star'}$ diagram (for certain choices
of vertices $y$ and $y'$ as described in the construction of definition
\ref{def: D_star_star'}). Cf. \emph{i)} in the example above, where
'splitting' at any vertex in the horizontal row of $x_{2}$ or in
$x'_{1}$ yields such a pair. Let us stress once more that we are
not interested in the combinatorics of such 'splittings', but only
in the classes of $D_{\star,\star'}$ diagrams, where $\star,\star'\in\{(\bigcirc,n)_{n\geq3},\square,\boxslash,\bot\}$,
themselves (cf. also Remarks \ref{rem:only class of B_star_B} and
\ref{rem:only class of C_B_B}). 
\end{rem}
As for types $\mathbb{B}^{(1)}$ and $\mathbb{C}^{(1)}$, mutating
these diagrams 'mostly' yields a diagram of this family again (Fact
\ref{lem: D^(1) mutation steps} 1) a)). We might consider these diagrams
as being in 'connected stage'. For width equal to zero, they may be
mutated into 'joined stage' (cf. \ref{lem: D^(1) mutation steps}
1) c)):

\begin{defn}
\label{def: D_star_vee_star} For $\star,\star'\in\{(\bigcirc,n)_{n\geq3},\square,\boxslash,\bot\}$,
$D_{\star\vee\star'}$ is the infinite graph obtained by the following
construction: 
\end{defn}
\begin{itemize}
\item if $(\star,\star')\neq\left((\bigcirc,n),(\bigcirc,m)\right)$ then
delete the $\nabla$ graphs attached to $x_{1},x'_{1}$ in $D_{\star}$
respectively $D_{\star'}$ without deleting $x_{1},x_{1}'$ themselves.
Identify $x_{1}$ and $x_{1}'$ and denote this vertex by $\bullet$. 
\item if $(\star,\star')=\left((\bigcirc,n),(\bigcirc,m)\right)$ then delete
the $\nabla$ graphs attached to $x_{1},x_{2}$ in $D_{(\bigcirc,n)}$
respectively $x'_{1},x'_{2}$ in $D_{(\bigcirc,m)}$ without deleting
these four vertices themselves. Identify the first and second vertex
of the following pairs: $(x_{1},a'_{1}),(x_{2},a'_{3}),$ $(a_{1},x'_{1}),(a_{3},x'_{2})$.
Finally identify $a_{2}$ and $a'_{2}$ and denote this vertex by
$\bullet$. 
\end{itemize}
The subgraph of $D_{\star\vee\star'}$ induced on the vertices of
$\star\cup\star'$ after identification is denoted by $\star\vee\star'$.
A $D_{\star\vee\star'}$ \emph{diagram} is a cyclically oriented simply-laced
diagram $\Gamma$ with underlying undirected graph equal to a connected
full subgraph of $D_{\star\vee\star'}$ containing the vertices of
$\star\vee\star'$ . In case of $(\star,\star')=((\bigcirc,n),(\bigcirc,m))$,
we additionally require that $|\Gamma_{0}|\geq6$.

\begin{example}
We give three examples, the first being a $D_{(\bigcirc,4)\vee(\bigcirc,3)}$,
the second a $D_{(\bigcirc,3)\vee\boxslash}$ and the last a $D_{\boxslash\vee\bot}$
diagram.

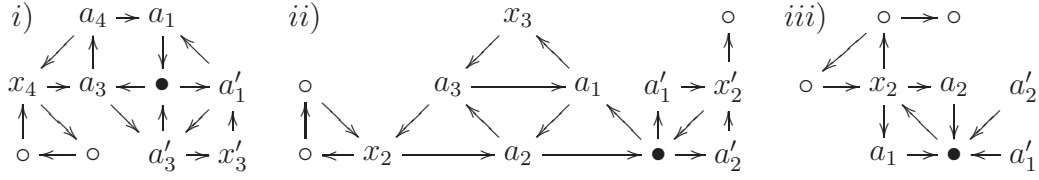
\[
\xyR{.6pc}\xyC{.7pc}\xymatrix{i) & a_{4}\ar[r]\ar[dl] & a_{1}\ar[d] &  & ii) &  &  & x_{3}\ar[dl] &  &  & \circ & iii) & \circ\ar[dl]\ar[r] & \circ\\
x_{4}\ar[r]\ar[dr] & a_{3}\ar[dr]\ar[u] & \bullet_{}\ar[l]\ar[r] & a'_{1}\ar[dl]\ar[ul] & \circ\ar[dr] &  & a_{3}\ar[rr]\ar[dl] &  & a_{1}\ar[dl]\ar[ul] & a'_{1}\ar[r] & x'_{2}\ar[dl]\ar[u] & \circ\ar[r] & x_{2}\ar[r]\ar[d]\ar[u] & a_{2}\ar[d] & a'_{2}\ar[dl]\\
\circ\ar[u] & \circ\ar[l] & a'_{3}\ar[u]\ar[r] & x'_{3}\ar[u] & \circ\ar[u] & x_{2}\ar[rr]\ar[l] &  & a_{2}\ar[ul]\ar[rr] &  & \bullet\ar[r]\ar[ul]\ar[u] & a'_{2}\ar[u] &  & a_{1}\ar[r] & \bullet\ar[ul] & a'_{1}\ar[l]}
\]

\end{example}
\noindent As already hinted at by the shape of $D_{(\bigcirc,n)\vee(\bigcirc,m)}$,
further mutation (cf. Fact \ref{lem: D^(1) mutation steps} 2)) might
yield diagrams which are in a sort of 'merged stage' :

\begin{defn}
1) \label{def: D_bigcirc_wedge_square}Let $n\in\mathbb{N}_{\geq3}$.
Then $D_{(\bigcirc,n)\wedge\square}$ is the infinite undirected graph
obtained by the following construction. Let $\{a_{1},..,a_{n-1}\}$
be an $A_{n-1}$ graph labeled linearly, $(\nabla_{i},x_{i})_{1\leq i\leq n-2}$
be $\nabla$ graphs, $\{\bullet_{j},a'_{j}\}_{1\leq j\leq2}$ be a
$A_{3}^{(1)}$ graph such that $a_{1}$ and $a_{2}$ are not connected.
For all $1\leq i\leq n-2$ join $x_{i}$ to both $a_{i}$ and $a_{i+1}$
by single edges and identify $\bullet_{1}$ with $a_{1}$ and $\bullet_{2}$
with $a_{n-1}$.

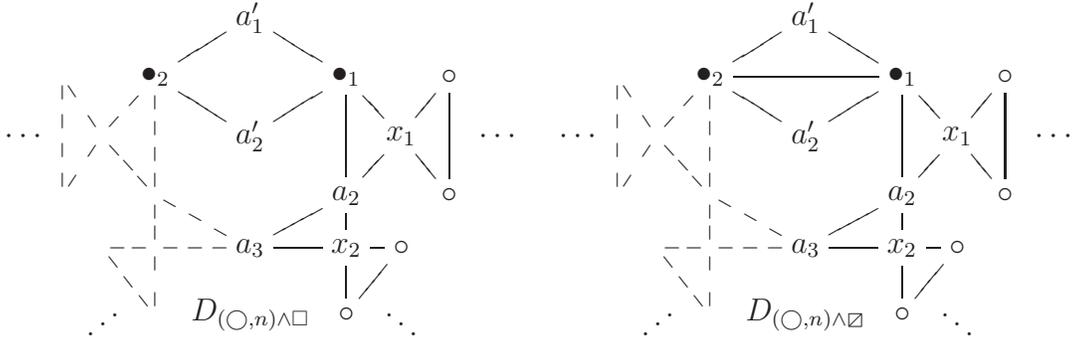
\[
\xyC{.1pc}\xyR{.5pc}\xymatrix{ &  &  &  & a'_{1}\ar@{-}'[dr]'[dd]'[dl]'[] &  &  &  &  &  &  &  &  &  & a'_{1}\ar@{-}'[dr]'[dd]'[dl]'[]\\
 & {} &  & \bullet_{2} &  & \bullet_{1} &  & \circ &  &  &  & {} &  & \bullet_{2}\ar@{-}[rr] &  & \bullet_{1} &  & \circ\\
\cdots &  & {}\ar@{--}'[ul]'[dl]'[]\ar@{--}'[ur]'[dr]'[] &  & a'_{2} &  & x_{1}\ar@{-}'[ur]'[dr]'[]\ar@{-}'[ul]'[dl]'[] &  & \cdots &  & \cdots &  & {}\ar@{--}'[ul]'[dl]'[]\ar@{--}'[ur]'[dr]'[] &  & a'_{2} &  & x_{1}\ar@{-}'[ur]'[dr]'[]\ar@{-}'[ul]'[dl]'[] &  & \cdots\\
 & {} &  & {} &  & a_{2} &  & \circ &  &  &  & {} &  & {} &  & a_{2} &  & \circ\\
 &  & {} & {}\ar@{--}'[l]'[d]'[]\ar@{--}'[u]'[r]'[] & a_{3} & x_{2}\ar@{-}'[r]'[d]'[]\ar@{-}'[u]'[l]'[] & \circ &  &  &  &  &  & {} & {}\ar@{--}'[l]'[d]'[]\ar@{--}'[u]'[r]'[] & a_{3} & x_{2}\ar@{-}'[r]'[d]'[]\ar@{-}'[u]'[l]'[] & \circ\\
 &  & \iddots & {} & D_{(\bigcirc,n)\wedge\square} & \circ & \ddots &  &  &  &  &  & \iddots & {} & D_{(\bigcirc,n)\wedge\boxslash} & \circ & \ddots}
\]
 By $(\bigcirc,n)\wedge\square$ we denote the subgraph of $D_{(\bigcirc,n)\wedge\square}$
induced on $\{a_{i}\}_{2\leq i\leq n-2}\cup\{\bullet_{j},a'_{j}\}_{1\leq j\leq2}$.
A \emph{$D_{(\bigcirc,n)\wedge\square}$ diagram} $\Gamma$ is a simply-laced
diagram with underlying undirected graph equal to a connected full
subgraph of $D_{(\bigcirc,n)\wedge\square}$ containing the vertices
of $(\bigcirc,n)\wedge\square$ such that the following condition
is satisfied: each full subgraph $\Gamma'$ of $\Gamma$ with underlying
undirected graph a cycle is cyclically oriented, except for the case
where $\Gamma'_{0}$ equals $\{a_{i}\}_{2\leq i\leq n-2}\cup\{\bullet_{1},\bullet_{2},a'_{2}\}$.\\

\noindent 2) \label{def: bigcirc_wedge_boxslash} For $n\in\mathbb{N}_{\geq3}$
let $D_{(\bigcirc,n)\wedge\boxslash}$ denote the infinite undirected
graph obtained by joining $\bullet_{1},\bullet_{2}$ in $D_{(\bigcirc,n+1),\square}$
by a single edge.

\noindent For any of $\star\in\{\boxslash,\overleftrightarrow{\boxslash}\}$,
the subgraph of $D_{(\bigcirc,n)\wedge\boxslash}$ induced on $\{a_{i}\}_{2\leq i\leq n-1}\cup\{\bullet_{j},a'_{j}\}_{1\leq j\leq2}$
is denoted by $(\bigcirc,n)\wedge\star$. A \emph{$D_{(\bigcirc,n)\wedge\star}$
diagram} is a simply-laced diagram $\Gamma$ with underlying undirected
graph equal to a connected full subgraph of $D_{(\bigcirc,n)\wedge\boxslash}$
containing the vertices of $(\bigcirc,n)\wedge\star$ such that the
following conditions are satisfied. 
\end{defn}
\begin{itemize}
\item If $\star=\boxslash$, then $\Gamma$ is cyclically oriented. 
\item If $\star=\overleftrightarrow{\boxslash}$, let $e$ denote the edge
between $\bullet_{1}$ and $\bullet_{2}$. Then the graph $(\Gamma_{0},\Gamma_{1}\setminus e)$
is cyclically oriented, while the subgraph induced on $\{\bullet_{1},\bullet_{2}\}\cup\{a_{i}\}_{2\leq i\leq n-1}$
is not. 
\end{itemize}
\begin{example}
The first example shows a $D_{(\bigcirc,3)\wedge\boxslash}$, the
second is a $D_{(\bigcirc,3)\wedge\overleftrightarrow{\boxslash}}$
diagram. We indicate possible enlargements by dots. \[
\xyR{.5pc}\xyC{.5pc}
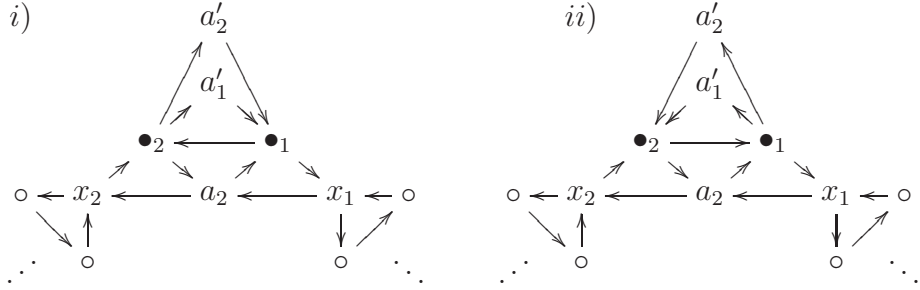
\xymatrix{i) &  &  & a'_{2}\ar[ddr] &  &  &  &  &  & ii) &  & a'_{2}\ar[ddl]\\
 &  &  & a'_{1}\ar[dr] &  &  &  &  &  &  &  & a'_{1}\ar[dl]\\
 &  & \bullet_{2}\ar[dr]\ar[ur]\ar[uur] &  & \bullet_{1}\ar[ll]\ar[dr] &  &  &  &  &  & \bullet_{2}\ar[dr]\ar[rr] &  & \bullet_{1}\ar[ul]\ar[uul]\ar[dr]\\
\circ\ar[dr] & x_{2}\ar[l]\ar[ur] &  & a_{2}\ar[ur]\ar[ll] &  & x_{1}\ar[ll]\ar[d] & \circ\ar[l] &  & \circ\ar[dr] & x_{2}\ar[l]\ar[ur] &  & a_{2}\ar[ur]\ar[ll] &  & x_{1}\ar[ll]\ar[d] & \circ\ar[l]\\
\iddots & \circ\ar[u] &  &  &  & \circ\ar[ur] & \ddots &  & \iddots & \circ\ar[u] &  &  &  & \circ\ar[ur] & \ddots}
\]

\end{example}
As we know from the discussion of type $\mathbb{D}$, an interesting
issue is the degeneration of the diagrams at hand if the length of
the circle goes to two. Indeed, we encounter new diagrams (cf. Fact
\ref{lem: D^(1) mutation steps} 3)):

\begin{defn}
\label{def: D_square_wegde_square} 1) \label{def: square_wedge_square}
$D_{\square\wedge\square}$ is the infinite graph obtained by the
following construction. Take the induced subgraph on $\square\subset D_{\square}$
and let $(\nabla,\bullet)$ be a $\nabla$ graph. Connect both $x_{1},x_{2}$
to $\bullet$.

By $\square\wedge\square$ we denote subgraph of $D_{\square\wedge\square}$
induced on the vertices $\{a_{i},x_{i}\}_{1\leq i\leq2}\cup\{\bullet\}$.
A $D_{\square\wedge\square}$ \emph{diagram} is a simply-laced diagram
$\Gamma$ with underlying undirected graph equal to a connected and
full subgraph of $D_{\square\wedge\square}$ containing the vertices
of $\square\wedge\square$ such that the following is fulfilled. 
\end{defn}
\begin{itemize}
\item For $1\leq i\leq2$, the subgraph induced on the set of vertices $\{x_{1},\bullet,x_{2},a_{i}\}$
is cyclically oriented. 
\item The subgraph induced on $\Gamma_{0}\setminus\{a_{i},x_{i}\}_{1\leq i\leq2}$
is cyclically oriented. \\

\end{itemize}
\noindent 2) \label{def: boxslash_wedge_boxslash} By $D_{\boxslash\wedge\boxslash}$
we denote the infinite graph obtained by inserting two edges between
$x_{1}$ and $x_{2}$ in $D_{\square\wedge\square}$.

The subgraph of $D_{\boxslash\wedge\boxslash}$ induced on the vertices
of $\{a_{i},x_{i}\}_{1\leq i\leq2}\cup\{\bullet\}$ is denoted by
$\boxslash\wedge\boxslash$. A $D_{\boxslash\wedge\boxslash}$ \emph{diagram}
is a cyclically oriented simply-laced diagram with underlying undirected
graph equal to a connected full subgraph of $D_{\boxslash\wedge\boxslash}$
containing the vertices of $\boxslash\wedge\boxslash$ .

\begin{example}
The first example depicts a $D_{\square\wedge\square}$, the second
a $D_{\boxslash\wedge\boxslash}$. We also include a $D_{(\bigcirc,3)\wedge\square}$
diagram to make the reader aware of the subtle difference in the definitions
involved. Again, possible enlargements are indicated by dots. \[
\xyC{1pc}\xyR{.3pc}
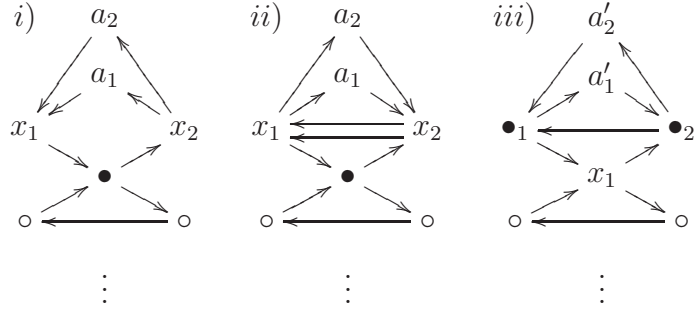
\xymatrix{i) & a{}_{2}\ar[ddl] &  & ii) & a{}_{2}\ar[ddr] &  & iii) & a'_{2}\ar[ddl]\\
 & a{}_{1}\ar[dl] &  &  & a{}_{1}\ar[dr] &  &  & a'_{1}\ar[dr]\\
x_{1}\ar[dr] &  & x_{2}\ar[ul]\ar[uul] & x_{1}\ar[dr]\ar[ur]\ar[uur] &  & x_{2}\ar@<.5ex>[ll]\ar@<-.5ex>[ll] & \bullet_{1}\ar[dr]\ar[ur] &  & \bullet_{2}\ar[ll]\ar[uul]\\
 & \bullet\ar[dr]\ar[ur] &  &  & \bullet\ar[dr]\ar[ur] &  &  & x_{1}\ar[dr]\ar[ur]\\
\circ\ar[ur] &  & \circ\ar[ll] & \circ\ar[ur] &  & \circ\ar[ll] & \circ\ar[ur] &  & \circ\ar[ll]\\
 & \vdots &  &  & \vdots &  &  & \vdots}
\]
Beyond this 'degenerated stage', there is one further stage which
yields new diagrams:
\end{example}
\begin{defn}
\label{def: D_boxtimes} $D_{\boxtimes}$ is the infinite graph obtained
by the following construction. Let $\{a_{i}\}{}_{1\leq i\leq4}$ be
an $A_{4}^{(1)}$ graph , $(\nabla,x_{1})$ a $\nabla$ graph. Join
each of $\{a_{i}\}_{1\leq i\leq4}$ with $x_{1}$ via one edge.$$ 
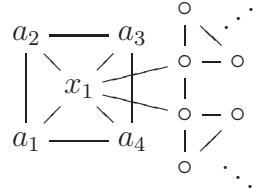
\begin{xy} 0;<.4pt,0pt>:<0pt,-.4pt>::  (75,75) *+{x_1} ="0", (25,125) *+{a_1} ="1", (25,25) *+{a_2} ="2", (125,25) *+{a_3} ="3", (125,125) *+{a_4} ="4", (175,50) *+{\circ} ="5", (175,100) *+{\circ} ="6", (175,0) *+{\circ} ="7", (225,50) *+{\circ} ="8", (225,100) *+{\circ} ="9", (175,150) *+{\circ} ="10", (225,0) *+{\iddots} ="11", (225,150) *+{\ddots} ="12", "1", {\ar@{-}"0"}, "0", {\ar@{-}"2"}, "3", {\ar@{-}"0"}, "0", {\ar@{-}"4"}, "0", {\ar@{-}"5"}, "6", {\ar@{-}"0"}, "2", {\ar@{-}"1"}, "4", {\ar@{-}"1"}, "2", {\ar@{-}"3"}, "4", {\ar@{-}"3"}, "5", {\ar@{-}"6"}, "5", {\ar@{-}"7"}, "8", {\ar@{-}"5"}, "6", {\ar@{-}"9"}, "10", {\ar@{-}"6"}, "7", {\ar@{-}"8"}, "9", {\ar@{-}"10"}, \end{xy} $$The
subgraph of $D_{\boxtimes}$ induced on the vertices of $\{a_{i}\}_{1\leq i\leq4}\cup\{x_{1}\}$
is denoted by $\boxtimes$. A $D_{\boxtimes}$ \emph{diagram} is a
cyclically oriented simply-laced diagram with underlying undirected
graph equal to a connected full subgraph of $D_{\boxtimes}$ containing
the vertices of $\boxtimes$. 
\end{defn}
\begin{rem}
Note that the families of $D_{\boxtimes}$,$D_{\star,\star'},D_{\star\vee\star'}$
and $D_{\star\wedge\star'}$ diagrams with $\star,\star'\in\bigl\{(\bigcirc,n)_{n\geq3},\square,\boxslash,$\linebreak$\bot,\overleftrightarrow{\boxslash}\bigr\}$
are mutually disjoint. Indeed, the crucial point is that any $D_{(\bigcirc,3)\vee(\bigcirc,3)}$
diagram $\Gamma$ (see \emph{i)} below for the general case) may be
distinguished from a $D_{\boxtimes}$ diagram. This is because the
first mentioned always contains an induced underlying subgraph of
the form \emph{ii)} since $|\Gamma_{0}|\geq6$. But a $D_{\boxtimes}$
diagram never does.

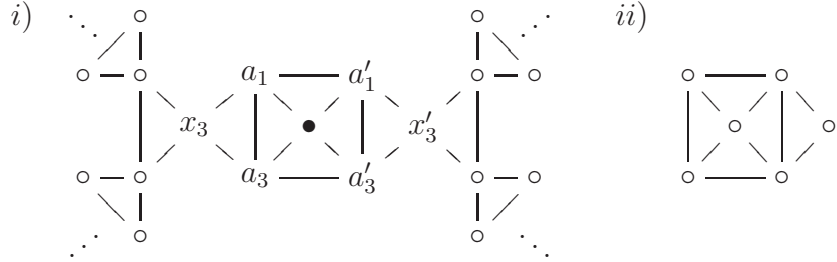
\[
\xyR{.05pc}\xyC{.4pc}\xymatrix{i) & \ddots & \circ\ar@{-}'[d]'[dl]'[] &  &  &  &  &  & \circ & \iddots &  & ii)\\
 & \circ & \circ &  & a_{1}\ar@{-}'[dl]'[dd]'[] &  & a'_{1} &  & \circ\ar@{-}'[u]'[r]'[] & \circ &  &  & \circ\ar@{-}'[dr]'[rr]'[]'[dd] &  & \circ\\
 &  &  & x_{3}\ar@{-}'[ul]'[dl]'[] &  & \bullet\ar@{-}'[ul]'[ur]'[] &  & x'_{3}\ar@{-}'[ur]'[dr]'[] &  &  &  &  &  & \circ &  & \circ\\
 & \circ & \circ\ar@{-}'[d]'[l]'[] &  & a_{3}\ar@{-}'[ur]'[rr]'[] &  & a'_{3}\ar@{-}'[ur]'[uu]'[] &  & \circ\ar@{-}'[d]'[r]'[] & \circ &  &  & \circ\ar@{-}'[ur]'[rr]'[] &  & \circ\ar@{-}'[uu]'[ur]'[]\\
 & \iddots & \circ &  &  &  &  &  & \circ & \ddots}
\]

\end{rem}

\subsection{Summary}

\label{subsection: summary}

As already mentioned at the beginning of section \ref{sec:The--diagrams},
the notation in Theorem \ref{Theorem (mutation classes)} is chosen
in a way to allow for a concise formulation of the statement. However,
it is imprecise in so far as some of the symbols (e.g. $B_{\bot\wedge\bot}$
or $D_{\bot\wedge\bot}$) carry no meaning at all. To make the statement
rigorous, let it be understood that all families of diagrams whose
existence is suggested by Theorem \ref{Theorem (mutation classes)}
and which have not been defined up until now are empty. In order to
clarify the situation, we collect all families of diagrams which are
nonempty and give precise references to their definitions below.

Moreover, we list the 'characteristic' subgraphs of the infinite graphs
that occurred in the previous definitions. As we will see in the next
section, these subgraphs are essential to the behaviour under mutation
of the diagrams involved. Thus it is very important to stick precisely
to the notations we introduced, once more depicted below. Finally,
note that $(\bigcirc,n)\wedge\star$ and $(\bigcirc,n)\wedge\overleftrightarrow{\star}$
coincide for both $\star\in\{B,\boxslash\}$ since these are really
just the subgraphs of the corresponding (undirected!) graphs. However,
the families of diagrams associated to them are different, because
diagrams are oriented.\\

\begin{itemize}
\item type $\mathbb{A}$: $A$ diagrams; Def. \ref{def: A}
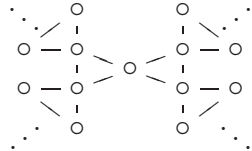
$$ \begin{xy} 0;<.4pt,0pt>:<0pt,-.3pt>:: (150,50) *+{\circ} ="0", (100,75) *+{\circ} ="1", (150,100) *+{\circ} ="2", (50,50) *+{\circ} ="3", (50,100) *+{\circ} ="4", (50,0) *+{\circ} ="5", (0,50) *+{\circ} ="6", (150,0) *+{\circ} ="7", (200,50) *+{\circ} ="8", (200,100) *+{\circ} ="9", (150,150) *+{\circ} ="10", (0,100) *+{\circ} ="11", (50,150) *+{\circ} ="12", (0,150) *+{\iddots} ="13", (0,0) *+{\ddots} ="14", (200,0) *+{\iddots} ="15", (200,150) *+{\ddots} ="16", "1", {\ar@{-}"0"}, "0", {\ar@{-}"2"}, "0", {\ar@{-}"7"}, "8", {\ar@{-}"0"}, "2", {\ar@{-}"1"}, "3", {\ar@{-}"1"}, "1", {\ar@{-}"4"}, "2", {\ar@{-}"9"}, "10", {\ar@{-}"2"}, "4", {\ar@{-}"3"}, "5", {\ar@{-}"3"}, "3", {\ar@{-}"6"}, "11", {\ar@{-}"4"}, "4", {\ar@{-}"12"}, "6", {\ar@{-}"5"}, "7", {\ar@{-}"8"}, "9", {\ar@{-}"10"}, "12", {\ar@{-}"11"}, \end{xy} $$\\

\item type $\mathbb{B}$: $B$ diagrams; Def. \ref{def:B} \[
\xyR{.01pc}\xyC{.4pc}
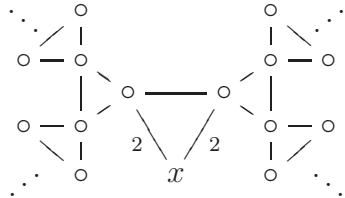
\xymatrix{\ddots & \circ &  &  &  & \circ & \iddots\\
\circ & \circ\ar@{-}'[l]'[u]'[]'[dr]'[dd] &  &  &  & \circ\ar@{-}'[u]'[r]'[]'[dl]'[dd] & \circ\\
 &  & \circ\ar@{-}[rr] &  & \circ\ar@{-}[ddl]^{2}\\
\circ & \circ\ar@{-}'[l]'[d]'[]'[uu] &  &  &  & \circ\ar@{-}'[r]'[d]'[]'[uu] & \circ\\
\iddots & \circ &  & x\ar@{-}[uul]^{2} &  & \circ & \ddots}
\]

\item type $\mathbb{D}$: $D_{\star}$ diagrams, $\star\in\bigl\{(\bigcirc,n)_{n\geq3},\square,\boxslash,\bot\bigr\}$;
Def. \ref{def:D} \[
\xyR{.55pc}\xyC{.48pc}
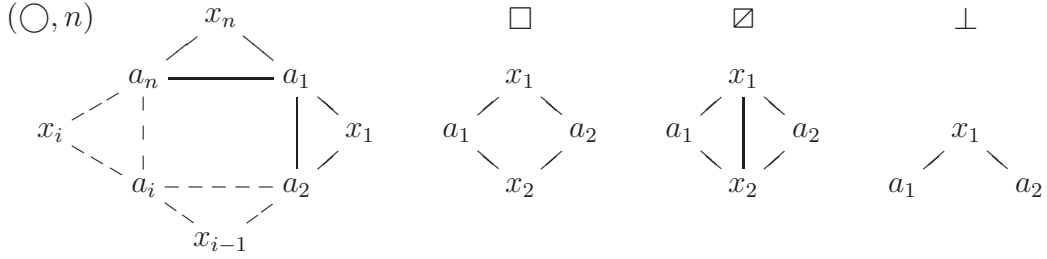
\xymatrix{(\bigcirc,n) &  & x_{n} &  &  &  &  & \square &  &  &  & \boxslash &  &  &  & \bot\\
 & a_{n}\ar@{--}'[dd]'[dl]'[] &  & a_{1}\ar@{-}'[ll]'[ul]'[] &  &  &  & x_{1}\ar@{-}'[dl]'[dd]'[dr]'[] &  &  &  & x_{1}\ar@{-}'[dl]'[dd]'[dr]'[]'[dd]\\
x_{i} &  &  &  & x_{1} &  & a_{1} &  & a_{2} &  & a_{1} &  & a_{2} &  &  & x_{1}\\
 & a_{i}\ar@{--}'[rr]'[dr]'[] &  & a_{2}\ar@{-}'[uu]'[ur]'[] &  &  &  & x_{2} &  &  &  & x_{2} &  &  & a_{1}\ar@{-}'[ur]'[rr] &  & a_{2}\\
 &  & x_{i-1}}
\]
 \\

\item type $\mathbb{B}^{(1)}$:

\begin{itemize}
\item $B_{\star,B}$ diagrams, $\star\in\{(\bigcirc,n)_{n\geq3},\square,\boxslash,\bot\}$;
Def. \ref{def: B_star_B} 
\item $B_{\star\wedge B}$ diagrams, $\star\in\{\square,\boxslash\}$; Def.
\ref{def: B_star_wedge_B}, 1) 
\item $B_{\boxslash\wedge\square}$ diagrams; Def. \ref{def: B_star_wedge_B},
2) 
\item $B_{(\bigcirc,n)\wedge\star}$ diagrams, $\star\in\{B,\overleftrightarrow{B}\}$;
Def. \ref{def: B_star_wedge_B}, 3) \[
\xyR{.7pc}\xyC{.2pc}
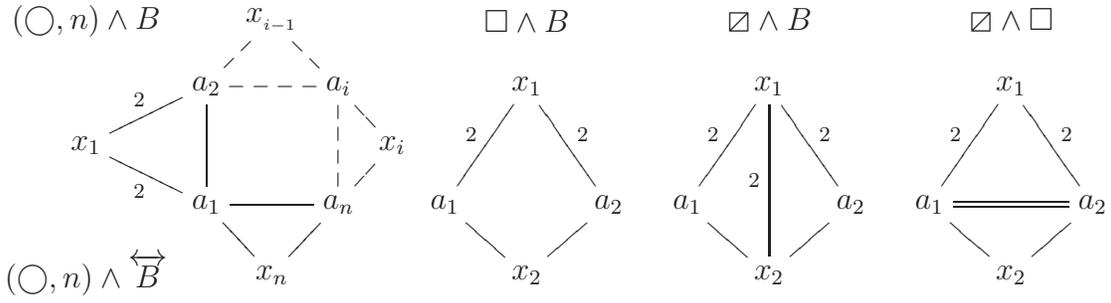
\xymatrix{(\bigcirc,n)\wedge B &  & x_{_{i-1}} &  &  &  & \square\wedge B &  &  &  & \boxslash\wedge B &  &  &  & \boxslash\wedge\square\\
 & a_{2}\ar@{--}'[ur]'[rr]'[] &  & a_{i}\ar@{--}'[dr]'[dd]'[] &  &  & x_{1}\ar@{-}[ddr]^{2}\ar@{-}[ddl]_{2}^{} &  &  &  & x_{1}\ar@{-}[ddr]^{2}\ar@{-}[ddl]_{2}^{} &  &  &  & x_{1}\ar@{-}[ddr]^{2}\ar@{-}[ddl]_{2}^{}\\
x_{1}\ar@{-}[ur]^{2} &  &  &  & x_{i}\\
 & a_{1}\ar@{-}[ul]^{2}\ar@{-}[uu] &  & a_{n}\ar@{-}'[dl]'[ll]'[] &  & a_{1}\ar@{-}'[dr]'[rr]^{} &  & a_{2} &  & a_{1}\ar@{-}'[dr]'[rr]^{} &  & a_{2} &  & a_{1}\ar@{-}'[dr]'[rr]^{} &  & a_{2}\ar@{=}[ll]\\
(\bigcirc,n)\wedge\overleftrightarrow{B} &  & x_{n} &  &  &  & x_{2} &  &  &  & x_{2}\ar@{-}[uuu]^{2} &  &  &  & x_{2}}
\]
 \\

\end{itemize}
\item type $\mathbb{C}^{(1)}$

\begin{itemize}
\item $C_{B,B}$ diagrams; Def. \ref{def: C_B_B} 
\item $C_{B\wedge B}$ diagrams; Def. \ref{def: C_B_wedge_B} \[
\xyR{.05pc}\xyC{.4pc}\xymatrix{B\wedge B & x_{1}\ar@{-}[rr]^{4} &  & x_{2}\ar@{-}[dl]^{2}\\
 &  & \bullet\ar@{-}[ul]^{2}}
\]
 \\

\end{itemize}
\item type $\mathbb{D}^{(1)}$

\begin{itemize}
\item $D_{\star,\star'}$ diagrams, $\star,\star'\in\{(\bigcirc,n)_{n\geq3},\square,\boxslash,\bot\}$;
Def. \ref{def: D_star_star'} 
\item $D_{\star\vee\star'}$ diagrams, $\star,\star'\in\{(\bigcirc,n)_{n\geq3},\square,\boxslash,\bot\}$;
Def. \ref{def: D_star_vee_star} 
\item $D_{(\bigcirc,n)\wedge\star}$ diagrams, $\star\in\{\square,\boxslash,\overleftrightarrow{\boxslash}\}$
and $n\in\mathbb{N}_{\geq3}$; Def. \ref{def: D_bigcirc_wedge_square} 
\item $D_{\star\wedge\star}$ diagrams, $\star\in\{\square,\boxslash\}$;
Def. \ref{def: D_square_wegde_square} 
\item $D_{\boxtimes}$ diagrams; Def. \ref{def: D_boxtimes} \[
\xyC{.2pc}\xyR{1pc}\xymatrix{ &  & (\bigcirc,n)\wedge\square &  &  &  &  & (\bigcirc,n)\wedge\boxslash &  & (\bigcirc,n)\wedge\overleftrightarrow{\boxslash}\\
 &  & a'_{1}\ar@{-}'[dr]'[dd]'[dl]'[] &  &  &  &  &  & a'_{1}\ar@{-}'[dr]'[dd]'[dl]'[]\\
 & \bullet_{2} &  & \bullet_{1}\ar@{-}'[dr]'[dd]'[] &  &  &  & \bullet_{2} &  & \bullet_{1}\ar@{-}'[dr]'[dd]'[]'[ll]\\
x_{n-2}\ar@{-}'[ur]'[dr]'[] &  & a'_{2} &  & x_{1} &  & x_{n-2}\ar@{-}'[ur]'[dr]'[] &  & a'_{2} &  & x_{1}\\
 & a_{n-2} &  & a_{2} &  &  &  & a_{n-2} &  & a_{2}\\
{} & x_{i+1}\ar@{--}'[u]'[r]'[] & a_{i} & x_{i}\ar@{--}'[u]'[l]'[] &  &  &  & x_{i+1}\ar@{--}'[u]'[r]'[] & a_{i} & x_{i}\ar@{--}'[u]'[l]'[]}
\]
 \\
 \[
\xymatrix{ & \square\wedge\square &  &  & \boxslash\wedge\boxslash &  &  &  & \boxtimes\\
 & a{}_{2}\ar@{-}[ddl]\ar@{-}[ddr] &  &  & a{}_{2}\ar@{-}[ddl]\ar@{-}[ddr] &  &  & a_{2}\ar@{-}'[rr]'[dr]'[] &  & a_{3}\ar@{-}'[dd]'[dl]'[]\\
 & a{}_{1} &  &  & a{}_{1} &  &  &  & x_{1}\\
x_{1} &  & x_{2} & x_{1} &  & x_{2}\ar@{=} &  & a_{1}\ar@{-}'[uu]'[ur]'[] &  & a_{4}\ar@{-}'[ll]'[ul]'[]\\
 & \bullet\ar@{-}'[ur]'[uu]'[ul]'[] &  &  & \bullet\ar@{-}'[ur]'[uu]'[ul]'[]}
\]

\end{itemize}
\end{itemize}

\section{Proof\label{sec:Proof}}

\noindent The idea for the proof is the same for all types. In the
first step we prove that the set of families of graphs defined for
a particular type is invariant under mutation. The second step then
applies these results to show that any diagram belonging to one of
the families defined for a particular type indeed is an element of
that type's mutation class.

Mutation is a local concept, hence the first step mentioned above
involves a detailed examination of the effect of diagram mutation
at a given vertex (depending on the local situation). The results
are presented in Facts $4.1.-4.7.$ We shall sketch the proofs for
Facts 4.1 and 4.2. only, leaving the rest to the reader. To get a
good feeling for the statements, we recommend the applet by Keller
\cite{Keller}.

\subsection{Type $\mathbb{A}$}

\begin{fact}
\label{Lemma: A mutation steps}\emph{Let $\Gamma$ be an $A$ diagram,
$k$ any vertex of $\Gamma$. Then $\mu_{k}(\Gamma)$ is an $A$ diagram.} 
\end{fact}
\begin{proof}
All cases to be considered are encoded by\emph{\[
\xyC{2pc}\xyR{.1pc}
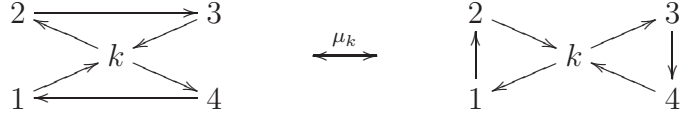
\xymatrix{2\ar[rr] &  & 3\ar[dl] &  &  & 2\ar[dr] &  & 3\ar[dd]\\
 & k\ar[ul]\ar[dr] &  & \ar@{<->}[r]^{\mu_{k}} &  &  & k\ar[dl]\ar[ur]\\
1\ar[ur] &  & 4\ar[ll] &  &  & 1\ar[uu] &  & 4\ar[ul]}
\]
}Indeed, if $\mbox{deg}(k)=4$, then the only case to be considered
is the one depicted above. If $\mbox{deg}(k)=3$, say $\mbox{deg}^{+}(k)=1$,
then the only possiblity for the local situation at $k$ in an $A$
diagram is of the form shown on the left hand side below\emph{\[
\xyC{2pc}\xyR{.1pc}
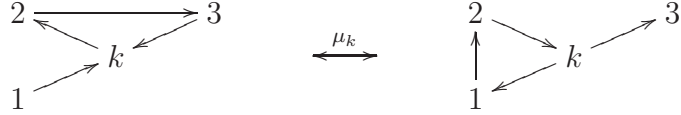
\xymatrix{2\ar[rr] &  & 3\ar[dl] &  &  & 2\ar[dr] &  & 3\\
 & k\ar[ul] &  & \ar@{<->}[r]^{\mu_{k}} &  &  & k\ar[dl]\ar[ur]\\
1\ar[ur] &  &  &  &  & 1\ar[uu]}
\]
}Note that performing mutation on the induced subdiagram has the same
effect as taking the induced subdiagram after applying mutation to
the general situation. That is, mutation commutes with taking induced
subdiagrams (this in fact follows directly from the definition of
mutation). Thus, correctness of the general case implies correctness
of all other cases indeed. 
\end{proof}
\begin{thm*}
\textbf{\emph{\ref{Theorem (mutation classes)}}} \emph{(1)} Let $\Gamma$
be a connected diagram. Then $\Gamma$ is of mutation type $\mathbb{A}$
if and only if it is an $A$ diagram. 
\end{thm*}
\begin{proof}
Since any diagram with underlying graph of Dynkin type $\mathbb{A}$
is itself an $A$ diagram, one direction follows from Fact \ref{Lemma: A mutation steps}.
For the other direction, we proceed by induction on $|\Gamma_{0}|$.

For $|\Gamma_{0}|=1$ there is nothing to show. So assume $\Gamma$
has $n+1$ vertices. If no vertex belongs to a cycle we are done.
Else choose a vertex $y$ such that the induced subgraph $\Gamma'$
on $\Gamma_{0}\setminus\{y\}$ is connected and apply the inductive
assumption to $\Gamma^{\prime}$. Denote by $\mu$ a series of mutations
transforming $\Gamma'$ into a diagram with underlying undirected
graph $A_{n}$. Then $\mu(\Gamma)$ has underlying undirected graph
$A_{n}$ with $y$ joined to some of its vertices, but it is also
an $A$ diagram by Fact \ref{Lemma: A mutation steps}. If $y$ is
connected to only one vertex, then the underlying graph of $\mu(\Gamma)$
must already be $A_{n+1}$. Assuming that $y$ is connected to more
than two vertices of $A_{n}$, we label the vertices of $A_{n}$ linearly
and choose the smallest three which are joined to $y$, say $i<j<k$\[
\xyR{1pc}\xyC{1pc}
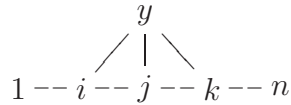
\xymatrix{ &  & y\ar@{-}[dl]\ar@{-}[d]\ar@{-}[dr]\\
1\ar@{--}'[r]'[rr]'[rrr]'[rrrr] & i & j & k & n}
\xyR{2pc}\]
 Since $A$ contains no cycles of length $m\geq3$, it follows that
$j=i+1$ and $k=i+2$. But then $\mu(\Gamma)$ is obviously no $A$
diagram. Therefore $y$ is connected to at most two vertices of $\mu(\Gamma)$,
which have to be neighbours as we just proved. Thus mutation in $y$
yields a diagram with underlying undirected graph $A_{n+1}$. 
\end{proof}

\subsection{Type $\mathbb{B}$}

\begin{fact}
\label{Lemma: B mutations steps}\emph{Let $\Gamma$ be an $B$ diagram,
$k$ any vertex of $\Gamma$. Then $\mu_{k}(\Gamma)$ is an $B$ diagram.} 
\end{fact}
\begin{proof}
Using the same arguments as in the proof of Fact \ref{Lemma: A mutation steps},
all cases to be considered beyond those already settled by Fact \ref{Lemma: A mutation steps}
are encoded by\emph{\[
\xyC{1.5pc}\xyR{.1pc}
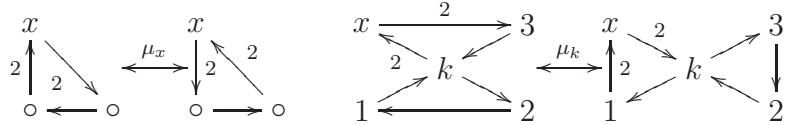
\xymatrix{x\ar[ddr]_{2} &  & x\ar[dd]_{}^{2} &  & x\ar[rr]^{2} &  & 3\ar[dl] & x\ar[dr]^{2} &  & 3\ar[dd]\\
 & \ar@{<->}[r]^{\mu_{x}} &  &  &  & k\ar[ul]^{2}\ar[dr] & \ar@{<->}[r]^{\mu_{k}} &  & k\ar[dl]\ar[ur]\\
\circ\ar[uu]^{2} & \circ\ar[l] & \circ\ar[r] & \circ\ar[uul]_{2} & 1\ar[ur] &  & 2\ar[ll] & 1\ar[uu]_{2}^{} &  & 2\ar[ul]}
\]
}  
\end{proof}
\begin{thm*}
\textbf{\emph{\ref{Theorem (mutation classes)}}} \emph{(2)} Let $\Gamma$
be a connected diagram. Then $\Gamma$ is of mutation type $\mathbb{B}$
if and only if it is a $B$ diagram. 
\end{thm*}
\begin{proof}
Use the same arguments as in the proof for type $\mathbb{A}$ with
the appropriate replacements. For the inductive step we additionally
assume that $y\neq x$, which is always possible. 
\end{proof}

\subsection{Type $\mathbb{D}$}

The following result is an easy exercise in diagram mutation (cf.
Definition \ref{def:D} for notations). For the first case, simply
use Fact \ref{Lemma: A mutation steps}.

\begin{fact}
\emph{\label{Lemma: D mutation steps} Let $\Gamma$ be a $D_{\star}$
diagram for $\star\in\bigl\{(\bigcirc,n)_{n\geq3},\square,\boxslash,\bot\bigr\}$
and $k$ be any vertex of $\Gamma$.}  
\end{fact}
\begin{enumerate}
\item If $k\notin\star_{0}$, then $\mu_{k}(\Gamma)$ is a $D_{\star}$
diagram.\\
 \\
 Else assume 
\item $\star=(\bigcirc,n)$. Then mutation in $k$ yields a

\begin{enumerate}
\item $D_{\square}$ diagram, if $n=3$, $k\in\{a_{i}\}_{1\leq i\leq3}$
and $x_{j}\in\Gamma_{0}$, where $x_{j}$ is the uniquely determined
vertex among $\{x_{i}\}_{1\leq i\leq3}$ not connected to $k$. 
\item $D_{\bot}$ diagram if $n=3$, $k\in\{a_{i}\}_{1\leq i\leq3}$ and
$x_{j}\notin\Gamma_{0}$ , $x_{j}$ as before. 
\item $D_{(\bigcirc,n-1)}$ diagram if $n>3$, $k\in\{a_{i}\}_{1\leq i\leq n}$. 
\item $D_{(\bigcirc,n+1)}$ diagram if $k\in\{x_{i}\}_{1\leq i\leq n}$.\\

\end{enumerate}
\item $\star=\square$. Then mutating in $k$ yields a

\begin{enumerate}
\item $D_{\bigcirc,3}$ diagram if $|\Gamma_{0}|>4$, $k\in\{x_{1},x_{2}\}$. 
\item $D_{\boxslash}$ diagram if $k\in\{a_{1},a_{2}\}$, or both $|\Gamma_{0}|=4$
and $k\in\{x_{1},x_{2}\}$\\

\end{enumerate}
\item $\star=\boxslash$. Then mutating in $k$ yields a

\begin{enumerate}
\item $D_{\square}$ diagram if $k\in\{a_{1},a_{2}\}$. 
\item $D_{\boxslash}$ diagram if $k=x_{i}$ for some $i\in\{1,2\}$ and
$k$ has a neighbour $y\in(\nabla_{i})_{0}$ such that the induced
subdiagram of $\Gamma$ on $\{y,x_{1},x_{2}\}$ carries non-linear
orientation and a $D_{\bot}$ diagram else.\\

\end{enumerate}
\item $\star=\bot$.

\begin{enumerate}
\item If $k\in\bot\setminus\{x\}$ , mutation in $x$ yields a $D_{\bot}$
diagram. 
\item Else $x=k$ and $\mu_{k}(\Gamma)$ is a

\begin{enumerate}
\item $D_{\boxslash}$ diagram if there exists a neighbour $y$ of $x$
such both $\{y,x,$\emph{$a_{i}\}_{1\leq i\leq2}$} carry linear orientation.
 
\item $D_{\bot}$ diagram if $x$ is of degree three and there exists a
neighbour $y\notin\bot_{0}$ of $x$ such both $\{y,x,a_{i}\}_{1\leq i\leq2}$
carry non-linear orientation. 
\item $D_{(\bigcirc,3)}$ diagram else.\\

\end{enumerate}
\end{enumerate}
\end{enumerate}
Before proving the Theorem for type $\mathbb{D}$, let us state the
following Lemma, which shall prove to be very convenient for the treatment
of types $\mathbb{B}^{(1)}$ and $\mathbb{D}^{(1)}$ as well.

\begin{lem}
\label{Lemma shrink cycle}Let $n\geq3$ and $C$ be a cyclically
oriented simply-laced diagram with underlying undirected graph $A_{n-1}^{(1)}$,
whose vertices are labeled by $\{1,...,n\}$ as depicted below, $\Gamma'$
be any diagram, $y\in\Gamma'_{0}$. Denote by $\Gamma$ the diagram
obtained by adding the arrows $\{y\rightarrow1\}$ and $\{n\rightarrow y\}$
of weight one. If $\mu:=\mu_{n-2}\circ\cdots\circ\mu_{2}\circ\mu_{1}$,
then $\mu(\Gamma)\setminus\Gamma'$ has underlying undirected graph
$D_{n}$ and $\{1\rightarrow y\}$ is the only element of $\Gamma_{1}$
between $(\mu(\Gamma)\setminus\Gamma')_{0}$ and $\Gamma_{0}'$ in
any direction. Moreover the induced subgraph on $\{y\}\cup\{1,..,n\}$
is simply-laced.

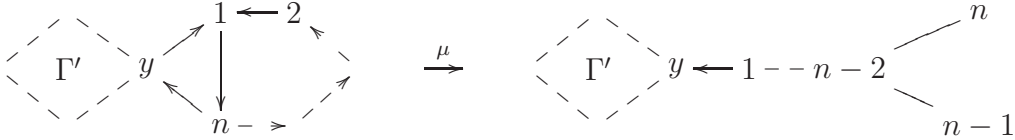
\[
\xyR{.5pc}\xyC{1.2pc}\xymatrix{ & {} &  & 1\ar[dd] & 2\ar[l] &  &  &  &  & {} &  &  &  & n\\
{}\ar@{--}'[dr]'[rr]'[ur]'[] & \Gamma' & y\ar[ur] &  &  & {}\ar@{-->}[ul] & {}\ar[r]^{\mu} & {} & {}\ar@{--}'[dr]'[rr]'[ur]'[] & \Gamma' & y & 1\ar[l]\ar@{--}[r] & n-2\\
 & {} &  & n\ar[ul]\ar@{-->}[r] & {}\ar@{-->}[ur] &  &  &  &  & {} &  &  &  & n-1\ar@{-}'[ul]'[uu]}
\]

\end{lem}
\begin{proof}
The statement follows by induction using the following two facts \[
\xyR{.5pc}\xyC{1.2pc}
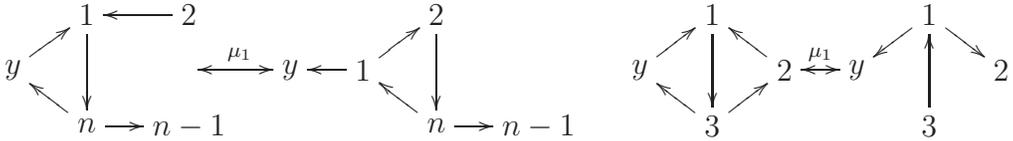
\xymatrix{ & 1\ar[dd] & 2\ar[l] &  &  & 2\ar[dd] &  &  & 1\ar[dd] &  &  & 1\ar[dl]\ar[dr]\\
y\ar[ur] &  & \ar@{<->}[r]^{\mu_{1}} & y & 1\ar[l]\ar[ur] &  &  & y\ar[ur] &  & 2\ar[ul]\ar@{<->}[r]^{\mu_{1}} & y &  & 2\\
 & n\ar[ul]\ar[r] & n-1 &  &  & n\ar[ul]\ar[r] & n-1 &  & 3\ar[ul]\ar[ur] &  &  & 3\ar[uu]}
\]

\end{proof}
Now we can easily prove:

\begin{thm*}
\textbf{\emph{\ref{Theorem (mutation classes)}}} \emph{(3)} Let $\Gamma$
be a connected diagram. Then $\Gamma$ is of mutation type $\mathbb{D}$
if and only if it is a $D_{\star}$ diagram, where $\star\in\bigl\{(\bigcirc,n)_{n\geq3},\square,\boxslash,\bot\bigr\}$. 
\end{thm*}
\begin{proof}
Any diagram with underlying undirected graph of Dynkin type $\mathbb{D}$
is a $D_{\bot}$ diagram, hence Fact \ref{Lemma: D mutation steps}
implies that any diagram of mutation type $\mathbb{D}$ is a $D_{\star}$
diagram for some $\star\in\bigl\{(\bigcirc,n)_{n\geq3},\square,\boxslash,\bot\bigr\}$.

To show the other direction, first assume that $|\Gamma_{0}|=4$.
Then $\star\notin\{(\bigcirc,n)_{n\geq3}\}$ by definition. Moreover,
by Fact \ref{Lemma: D mutation steps} (3) (b) and (4) (b) we may
assume that $\star=\bot$. Thus $\Gamma$ has underlying undirected
graph $D_{4}$.

For $|\Gamma_{0}|>4$ we first show that any $D_{\star}$ diagram
with $\star\in\{\bot,\boxslash,\square\}$ is mutation equivalent
to a $D_{(\bigcirc,m)}$ diagram, $m\leq|\Gamma_{0}|$. By Fact \ref{Lemma: D mutation steps}
(5) (b) we may restrict to $\star\in\{\boxslash,\square\}$; here
we use that if we are in the situation of subcase (5) (b) (ii), then
$\mu_{a_{1}}(\Gamma)$ is as described in one of the other subcases
of (5) (b). Now (4) (a) of the same Fact allows for further restriction
to $\star=\square$, hence we may apply (3) (a).

Thus we may assume that $\Gamma$ is a $D_{(\bigcirc,m)}$ diagram,
$m\leq|\Gamma_{0}|$. Set $n:=|\Gamma_{0}|$ and apply Fact \ref{Lemma: D mutation steps}
(2) (d) until $\Gamma$ is a cycle on $n$ vertices. We infer from
Lemma \ref{Lemma shrink cycle} that $\Gamma$ is mutation equivalent
to a diagram with underlying undirected graph $D_{n}$. 
\end{proof}

\subsection{Type $\mathbb{B}^{(1)}$}

Besides the difficulty of keeping track of notations (cf. subsection
\ref{sub:The-diagrams-for B^(1)}), the next Fact is a straightforward
exercise in diagram mutation. In order not to get puzzled by the technical
details of the first case treated below, we recommend to read the
paragraph directly before Definition \ref{def: B_star_wedge_B} once
more.

\begin{fact}
\label{lem: B^(1) mutation steps} \emph{The following is true.} 
\end{fact}
\begin{enumerate}
\item Let $\Gamma$ be a $B_{\star,B}$ diagram, where $\star\in\{(\bigcirc,n)_{n\geq3},\square,\boxslash,\bot\}$.
Recall that the endpoints of $\omega_{\Gamma}$ are labeled $x_{i}$
for some $i$ and $x'$. Let $\Gamma_{\setminus x'}$ denote the subgraph
of $\Gamma$ induced on $\Gamma_{0}\setminus\{x'\}$. Then $\Gamma_{\setminus x'}$
is a $D_{\star}$ diagram and a full connected subdiagram of $\Gamma$.
For $k$ any vertex of $\Gamma$ we distinguish between two cases.
If $k\in(\Gamma_{\setminus x'})_{0}$ then $\mu_{k}(\Gamma_{\setminus x'})$
is a $D_{\star'}$ diagram for some $\star'\in\bigl\{(\bigcirc,n)_{n\geq n},\square,\boxslash,\bot\bigr\}$
by Fact \ref{Lemma: D mutation steps}. Otherwise we simply set $\star'=\star$.
Then the following is true:

\begin{enumerate}
\item if $w(\Gamma)>0$, then $\mu_{k}(\Gamma)$ is a $B_{\star',B}$ diagram
and \emph{$|w(\mu_{k}(\Gamma))-w(\Gamma)|\leq1$ .} 
\item if $w(\Gamma)=0$ and $k\neq x_{i}$, then $\mu_{k}(\Gamma)$ is a
$B_{\star',B}$ diagram and \emph{$w(\Gamma)\leq w(\mu_{k}(\Gamma))\leq w(\Gamma)+1$}
. 
\item otherwise $w(\Gamma)=0$ and $k=x_{i}$ which is a neighbour of $x'$.
Assume that $\star$ equals

\begin{enumerate}
\item $(\bigcirc,n)$. Then $\mu_{k}(\Gamma)$ is a $B_{(\bigcirc,n+1)\wedge B}$
diagram. 
\item $\square$. Then $\mu_{k}(\Gamma)$ is a $B_{(\bigcirc,3)\wedge B}$
diagram. 
\item $\boxslash$. Then $\mu_{k}(\Gamma)$ is a

\begin{enumerate}
\item $B_{\boxslash\wedge B}$ diagram if $\{x',x{}_{1},x{}_{2}\}$ carries
non-linear orientation. 
\item $B_{\boxslash,B}$ diagram of width zero if $\{x',x{}_{1},x{}_{2}\}$
carries linear orientation and $\mbox{deg}(x)=5$. 
\item $B_{\bot,B}$ diagram of width zero else. 
\end{enumerate}
\item $\bot$. Then $\mu_{k}(\Gamma)$ is a

\begin{enumerate}
\item $B_{\boxslash\wedge B}$ diagram if both $\{x',x,a_{i}\}_{1\leq i\leq2}$
carry linear orientation. 
\item $B_{\boxslash,B}$ diagram of width zero if both $\{x_{},x',a_{i}\}_{1\leq i\leq2}$
carry non-linear orientation and $\mbox{deg}(x)=4$. 
\item $B_{\bot,B}$ diagram of width zero if both $\{x_{},x',a_{i}\}_{1\leq i\leq2}$
carry non-linear orientation and $\mbox{deg}(x)=3$. 
\item $B_{(\bigcirc,3)\wedge B}$ diagram else.\\

\end{enumerate}
\end{enumerate}
\end{enumerate}
\item Let $\Gamma$ be a $B_{(\bigcirc,n)\wedge\star}$ diagram, $\star\in\{B,\overleftrightarrow{B}\}$.

\begin{enumerate}
\item If $k\notin(\bigcirc,n)\wedge B$ then $\mu_{k}(\Gamma)$ is a $B_{(\bigcirc,n)\wedge\star}$
diagram. 
\item If $k=x_{1}$ then $\mu_{k}(\Gamma)$ is a $B_{(\bigcirc,n)\wedge\overleftrightarrow{B}}$
diagram if $\star=B$ and a $B_{(\bigcirc,n)\wedge B}$ diagram if
$\star=\overleftrightarrow{B}$. 
\item If $k\in\{x_{i}\}_{i\neq1}$ then $\mu_{k}(\Gamma)$ is a $B_{(\bigcirc,n+1)\wedge\star}$
diagram. 
\item If $k\in\{a_{i}\}_{i\neq1,2}$ and $n>3$, then $\mu_{k}(\Gamma)$
is a $B_{(\bigcirc,n-1)\wedge\star}$ diagram. 
\item If $n=3$ and $k=a_{3}$, then $\mu_{k}(\Gamma)$ is a $B_{\square\wedge B}$
diagram if $\star=B$. Else $\star=\overleftrightarrow{B}$ and $\mu_{k}$
yields a $B_{\boxslash\wedge\square}$ diagram. 
\item Else $k\in\{a_{1},a_{2}\}$. Assume $\star$ equals

\begin{enumerate}
\item $\overleftrightarrow{B}$. Then $\mu_{k}(\Gamma)$ is a $B_{(\bigcirc,n)\wedge\overleftrightarrow{B}}$
diagram. 
\item $B$. Then

\begin{enumerate}
\item for $n>3$, $\mu_{k}(\Gamma)$ is a $B_{(\bigcirc,n-1),B}$ diagram
of width zero. 
\item for $n=3$, $\mu_{k}(\Gamma)$ is a $B_{\square,B}$ diagram of width
zero if $x_{i}\in\Gamma_{0}$, where $x_{i}$ is the uniquely determined
vertex among $\{x_{j}\}_{1\leq j\leq3}$ not connected to $k$ and
a $B_{\bot,B}$ diagram of width zero else.\\

\end{enumerate}
\end{enumerate}
\end{enumerate}
\item Let $\Gamma$ be a $B_{\star\wedge\star'}$ diagram, $\star,\star'\in\{\square,\boxslash,B\}$.

\begin{enumerate}
\item If $k\notin(\star\wedge\star')_{0}$ then $\mu_{k}(\Gamma)$ is a
$B_{\star\wedge\star'}$ diagram. 
\item Else $k\in(\star\wedge\star')_{0}$.

\begin{enumerate}
\item If $k=x_{2}$ then $\mu_{k}(\Gamma)$ is a

\begin{enumerate}
\item $B_{(\bigcirc,3)\wedge B}$ diagram if $\star\wedge\star'=\square\wedge B$. 
\item $B_{\boxslash,B}$ diagram of width zero if $\star\wedge\star'=\boxslash\wedge B$
and there exists a neighbour $y\notin(\boxslash\wedge B)_{0}$ of
$x_{2}$ such that $\{x_{1},x_{2},y\}$ carries non-linear orientation. 
\item $B_{\bot,B}$ of width zero if $\star\wedge\star'=\boxslash\wedge B$
and there exists a neighbour $y\notin(\boxslash\wedge B)_{0}$ of
$x_{2}$ such that $\{x_{1},x_{2},y\}$ carries linear orientation. 
\item $B_{(\bigcirc,3)\wedge\overleftrightarrow{B}}$ diagram if $\star\wedge\star'=\boxslash\wedge\square$. 
\end{enumerate}
\item If $k=x_{1}$ then $\mu_{k}(\Gamma)$ is a

\begin{enumerate}
\item $B_{(\boxslash\wedge\square)}$ diagram if $\star\wedge\star'=\square\wedge B$
and vice versa. 
\item $B_{(\boxslash\wedge B)}$ diagram if $\star\wedge\star'=\boxslash\wedge B$. 
\end{enumerate}
\item Else $k\in\{a_{1},a_{2}\}$ and $\mu_{k}(\Gamma)$

\begin{enumerate}
\item $B_{(\boxslash\wedge B)}$ diagram if $\star\wedge\star'=\square\wedge B$
and vice versa. 
\item $B_{(\boxslash\wedge\square)}$ if $\star\wedge\star'=\boxslash\wedge\square$.\\

\end{enumerate}
\end{enumerate}
\end{enumerate}
\end{enumerate}
The idea for the difficult part of the proof of Theorem \ref{Theorem (mutation classes)}
(4) is to apply the above result to show that any diagram belonging
to one of the families defined in subsection \ref{sub:The-diagrams-for B^(1)}
is of mutation type $\mathbb{B}^{(1)}$. This is done by reducing
successively the cases to be considered, until we may apply Lemma
\ref{Lemma shrink cycle}.

\begin{thm*}
\textbf{\emph{\ref{Theorem (mutation classes)}}} \emph{(4)} Let $\Gamma$
be a connected diagram. Then $\Gamma$ is of mutation type \emph{$\mathbb{B}^{(1)}$}
if and only if it is a $B_{\star,B},B_{\star\wedge\star'}$ diagram,
where $\star\in\bigl\{(\bigcirc,n)_{n\geq3},\square,\boxslash,\bot\bigr\}$
and $\star'\in\{B,\overleftrightarrow{B}\}$, or a $B_{\square\wedge\boxslash}$
diagram. 
\end{thm*}
\begin{proof}
Any diagram with underlying undirected graph of Dynkin type $\mathbb{B}^{(1)}$
is a $B_{\bot,B}$ diagram, hence Fact \ref{lem: B^(1) mutation steps}
implies that any diagram of mutation type \emph{$\mathbb{B}^{(1)}$}
is a $B_{\star,B},B_{\star\wedge\star'}$ diagram, where $\star\in\bigl\{(\bigcirc,n)_{n\geq3},\square,\boxslash,\bot\bigr\}$
and $\star'\in\{B,\overleftrightarrow{B}\}$, or a $B_{\square\wedge\boxslash}$
diagram.

To show the other direction, first assume that $|\Gamma_{0}|=4$.
This restricts the possibilities for $\Gamma$ to $B_{\star\wedge\star'}$
diagrams, where $\star\in\bigl\{(\bigcirc,3),\square,\boxslash\bigr\}$
and $\star'\in\{B,\overleftrightarrow{B}\}$, or $B_{\boxslash\wedge\square}$,
or $B_{\bot,B}$ diagrams. Note that the last mentioned is a $B_{3}^{(1)}$
graph already. Thus Fact \ref{lem: B^(1) mutation steps} (3) (b)
(i) (C) implies that any $B_{\boxslash\wedge B}$ diagram is of mutation
type $\mathbb{B}^{(1)}$. Hence we need only deal with $B_{\star\wedge\star'}$
diagrams, where $\star\in\bigl\{(\bigcirc,3),\square,\boxslash\bigr\}$
and $\star'\in\{B,\overleftrightarrow{B}\}$. Now Fact \ref{lem: B^(1) mutation steps}
(2) (e) allows to reduce to the case of $B_{\square\wedge B}$ and
$B_{\boxslash\wedge\square}$ diagrams. But any of these are mutation
equivalent to a $B_{\boxslash\wedge B}$ diagram by (3) (b) (iii)
(A) and (3) (b) (ii) (B) respectively. Since we already proved $B_{\boxslash\wedge B}$
diagrams to be of mutation type $\mathbb{B}^{(1)}$, this finishes
the proof for $|\Gamma_{0}|=4$.

For $|\Gamma_{0}|>4$ we apply the results of Fact \ref{lem: B^(1) mutation steps}
in order to reduce to the case where $\Gamma$ is a $B_{(\bigcirc,m)\wedge B}$
diagram for some $m\leq|\Gamma_{0}|-2$. First note that any $B_{\star,B}$
diagram, where $\star\in\bigl\{(\bigcirc,m)_{m\geq3},\square,\boxslash,\bot\bigr\}$,
is mutation equivalent to a $B_{\star',B}$ diagram, where $\star'\in\bigl\{(\bigcirc,m)_{m\geq3},\square,\boxslash,\bot\bigr\}$,
of width zero. Next Fact \ref{lem: B^(1) mutation steps} (1) (c)
reduces to the case of $B_{\star\wedge\star'}$ diagrams, where $\star\in\bigl\{(\bigcirc,m)_{m\geq3},\square,\boxslash\bigr\}$
and $\star'\in\{B,\overleftrightarrow{B}\}$, or $B_{\boxslash\wedge\square}$
diagrams. In detail, if $\Gamma$ is a $B_{\square,B}$ diagram, apply
(1) (c) (ii). If $\Gamma$ is a $B_{\boxslash,B}$ diagram, mutate
in $a_{1}$ to get a $B_{\square,B}$ diagram of width zero. Finally,
if $\Gamma$ is a $B_{\bot,B}$ diagram, then mutation in $a_{1}$
allows to choose between the subcases of Fact \ref{lem: B^(1) mutation steps}
(1) (c) (iv), hence we may assume to be in subcase (1) (c) (iv) (D).
Now that we have restricted to $B_{\star\wedge\star'}$ diagrams,
where $\star\in\bigl\{(\bigcirc,m)_{m\geq3},\square,\boxslash\bigr\}$
and $\star'\in\{B,\overleftrightarrow{B}\}$, or $B_{\boxslash\wedge\square}$
diagrams, Fact \ref{lem: B^(1) mutation steps} (3) (b) reduces to
the case of $\Gamma$ being a $B_{(\bigcirc,m)\wedge\star}$ diagram,
where $\star\in\{B,\overleftrightarrow{B}\}$. Once more in detail,
use (3) (b) (i) (A) if $\Gamma$ is a $B_{\square\wedge B}$ diagram,
(3) (b) (i) (D) if $\Gamma$ is a $B_{\boxslash\wedge\square}$ diagram
and (3) (b) (iii) (A) to mutate any $B_{\boxslash\wedge B}$ diagram
into a $B_{\square\wedge B}$ diagram. Finally we are dealing with
$B_{(\bigcirc,m)\wedge\star}$ diagrams with $\star\in\{B,\overleftrightarrow{B}\}$
only. Now use Fact \ref{lem: B^(1) mutation steps} (2) (b) to exclude
$\star=\overleftrightarrow{B}$.

Thus we may assume that $\Gamma$ is a $B_{(\bigcirc,m)\wedge B}$
diagram, $m\leq|\Gamma_{0}|-2$. Set $n:=|\Gamma_{0}|$ and apply
Fact \ref{lem: B^(1) mutation steps} (2) (c) until $\Gamma$ is a
$B_{(\bigcirc,n-1)\wedge B}$ diagram. Using (2) (f) (ii) (A) of the
same Fact, we are in the situation of Lemma \ref{Lemma shrink cycle}
with $\Gamma'$ a diagram with underlying undirected graph $B_{2}$,
having $y$ as an end vertex. Thus $\Gamma$ is of mutation type $\mathbb{B}^{(1)}$. 
\end{proof}

\subsection{Type $\mathbb{C}^{(1)}$}

\begin{fact}
\label{lem: C^(1) mutation steps}\emph{The following is an easy exercise
in diagram mutation:} 
\end{fact}
\begin{enumerate}
\item Let $\Gamma$ be a $C_{B,B}$ diagram.

\begin{enumerate}
\item If $w(\Gamma)>0$, then $\mu_{k}(\Gamma)$ is a $C_{B,B}$ diagram
and \emph{$|w(\mu_{k}(\Gamma))-w(\Gamma)|\leq1$ .} 
\item Else $w(\Gamma)=0$ and $x,x'$ have a common neighbour $\bullet$.
For $k$ any vertex, $\mu_{k}(\Gamma)$ is a

\begin{enumerate}
\item $C_{B\wedge B}$ diagram if $k=\bullet$ and $\{x,\bullet,x'\}$ carries
linear orientation. 
\item $C_{B,B}$ diagram else. 
\end{enumerate}
\end{enumerate}
\item Let $\Gamma$ be a $C_{B\wedge B}$ diagram, $k$ any vertex of $\Gamma$.
Then $\mu_{k}(\Gamma)$ is a

\begin{enumerate}
\item $C_{B\wedge B}$ diagram if $k\neq\bullet$. 
\item $C_{B,B}$ diagram of width zero else. 
\end{enumerate}
\end{enumerate}
\begin{thm*}
\textbf{\emph{\ref{Theorem (mutation classes)}}} \emph{(5)} Let $\Gamma$
be a connected diagram. Then $\Gamma$ is of mutation type $\mathbb{C}^{(1)}$
if and only if it is a $C_{B,B}$ or $C_{B\wedge B}$ diagram. 
\end{thm*}
\begin{proof}
Since any diagram with underlying undirected graph of Dynkin type
$\mathbb{C}$ is itself an $C_{B,B}$ diagram, one direction follows
from Fact \ref{lem: C^(1) mutation steps}. For the other direction,
we proceed by induction on $|\Gamma_{0}|$.

For $|\Gamma_{0}|=3$ note that Fact \ref{lem: C^(1) mutation steps}
(2) (b) reduces to the case where $\Gamma$ is a $C_{B,B}$ diagram,
which has underlying graph $C_{2}^{(1)}$ then. Assume that $\Gamma$
has $n+1$ vertices and that the statement is true for $n$. If $\Gamma$
does not have underlying graph $C_{n}^{(1)}$ already, let $y\in\Gamma_{0}\setminus\{x,x'\}$
such that the induced subgraph on $\Gamma_{0}\setminus\{y\}$ is connected.
The rest is analogous to the proof for types $\mathbb{A}$ and $\mathbb{B}$. 
\end{proof}

\subsection{Type $\mathbb{D}^{(1)}$}

Again, the only difficulty for the next statement lies in keeping
track of notations. However, we strongly recommend the reader to work
through the discussion for type $\mathbb{B}^{(1)}$ (Fact \ref{lem: B^(1) mutation steps})
first, as this already conveys the right understanding for the situation
treated in the first case of the following

\begin{fact}
\label{lem: D^(1) mutation steps}\emph{The following is a straightforward
exercise in diagram mutation (cf. subsection \ref{sub:The--diagrams for D^(1)}
for notations).}  
\end{fact}
\begin{enumerate}
\item Let $\Gamma$ be a $D_{\star,\star'}$ diagram, where $\star,\star'\in\bigl\{(\bigcirc,n)_{n\geq3},\square,\boxslash,\bot\bigr\}$.
Recall that the endpoints of $\omega_{\Gamma}$ are labeled $x_{i}$
for some $i$ and $x'_{j}$ for some $j$. Let $\Gamma_{x_{i}}$ denote
the connected component of the subgraph of $\Gamma$ induced on $\Gamma_{0}\setminus\star'_{0}$
which is uniquely determined by containing $x_{i}$. Then $\Gamma_{x_{i}}$
is a $D_{\star}$ diagram and a full connected subdiagram of $\Gamma$.
For $k$ any vertex of $\Gamma$ we assume without loss of generality
that $k\in(\Gamma_{x_{i}})_{0}$. Then $\mu_{k}(\Gamma_{x_{i}})$
is a $D_{\star''}$ diagram for some $\star''\in\bigl\{(\bigcirc,n)_{n\geq n},\square,\boxslash,\bot\bigr\}$
by Fact \ref{Lemma: D mutation steps} and the following is true:

\begin{enumerate}
\item if $w(\Gamma)>0$, then $\mu_{k}(\Gamma)$ is a $D_{\star'',\star'}$
diagram and \emph{$|w(\mu_{k}(\Gamma))-w(\Gamma)|\leq1$.} 
\item if $w(\Gamma)=0$ and $k\neq x_{i}$, then $\mu_{k}(\Gamma)$ is a
$D_{\star'',\star'}$ diagram and \emph{$w(\Gamma)\leq w(\mu_{k}(\Gamma))\leq w(\Gamma)+1$.} 
\item otherwise $w(\Gamma)=0$ and $k=x_{i}$. Assume that $\star$ equals

\begin{enumerate}
\item $(\bigcirc,n)$. Then $\mu_{k}(\Gamma)$ is a $D_{(\bigcirc,n+1)\vee(\bigcirc,m+1)}$
diagram if $\star'=(\bigcirc,m)$ and a $D_{(\bigcirc,n+1)\vee\star'}$
diagram else. 
\item $\square$. Then $\mu_{k}(\Gamma)$ is a $D_{(\bigcirc,3),(\bigcirc,n)}$
diagram of width zero if $\star'=(\bigcirc,n)$ and a $D_{(\bigcirc,3)\vee\star'}$
diagram else. 
\item $\boxslash$. Then $\mu_{k}(\Gamma)$ is a

\begin{enumerate}
\item $D_{\boxslash\vee\star'}$ diagram if $\{x',x{}_{1},x{}_{2}\}$ carries
non-linear orientation. 
\item $D_{\boxslash,\star'}$ diagram of width zero if $\{x',x{}_{1},x{}_{2}\}$
carries linear orientation and $\mbox{deg}(x)=5$. 
\item $D_{\bot,\star'}$ diagram of width zero else 
\end{enumerate}
\item $\bot$. Then $\mu_{k}(\Gamma)$ is a

\begin{enumerate}
\item $D_{\boxslash\vee\star'}$ diagram if both $\{x',x,a_{i}\}_{1\leq i\leq2}$
carry linear orientation. 
\item $D_{\boxslash,\star'}$ diagram of width zero if both $\{x'_{},x,a_{i}\}_{1\leq i\leq2}$
carry non-linear orientation and $\mbox{deg}(x)=5$. 
\item $D_{\bot,\star'}$ diagram of width zero if both $\{x'_{},x,a_{i}\}_{1\leq i\leq2}$
carry non-linear orientation and $\mbox{deg}(x)=4$. 
\item $D_{(\bigcirc,3)\vee\star'}$ diagram else.\\

\end{enumerate}
\end{enumerate}
\end{enumerate}
\item Let $\Gamma$ be a $D_{\star\vee\star'}$ diagram, where $\star,\star'\in\bigl\{(\bigcirc,n)_{n\geq n},\square,\boxslash,\bot\bigr\}$,
$k$ a vertex of $\Gamma$.

\begin{enumerate}
\item If $k\neq\bullet$ then for

\begin{enumerate}
\item $(\star\vee\star')\neq((\bigcirc,n)\vee(\bigcirc,m))$ we may assume
without loss of generality that $k\notin\star'_{0}$. Denote by $\Gamma_{1}$
the connected component of the subgraph of $\Gamma$ induced on $\Gamma_{0}\setminus\star_{0}'$
containing $k$, so that $\Gamma_{1}$ is a $D_{\star}$ diagram.
Then $\mu_{k}(\Gamma_{1})$ is a $D_{\star''}$ diagram for some $\star''\in\bigl\{(\bigcirc,n)_{n\geq n},\square,\boxslash,\bot\bigr\}$
by Fact \ref{Lemma: D mutation steps}. Moreover, $\mu_{k}(\Gamma)$
is either one of the following

\begin{enumerate}
\item a $D_{\star'',\star'}$ diagram of width zero. 
\item or a $D_{\star''\vee\star'}$ diagram. 
\end{enumerate}
\item $(\star\vee\star')=((\bigcirc,n)\vee(\bigcirc,m))$ and $k\in\{a_{1},a_{3}\}$,
$\mu_{k}(\Gamma)$ is a $D_{(\bigcirc,m+1)\wedge\square}$ diagram
if $n=3$ and a $D_{(\bigcirc,n-1)\vee(\bigcirc,m+1)}$ diagram else. 
\end{enumerate}
\item Else $k=\bullet$. Assume that $(\star\vee\star')$ equals

\begin{enumerate}
\item $((\bigcirc,3)\vee(\bigcirc,3))$. Then $\mu_{k}(\Gamma)$ is a $D_{\square\vee\square}$
diagram if both $x_{3},x'_{3}\in\Gamma_{0}$ and a $D_{\square\vee\bot}$
diagram else. 
\item $((\bigcirc,n)\vee(\bigcirc,3))$ with $n\neq3$. Then $\mu_{k}(\Gamma)$
is a $D_{(\bigcirc,n-1)\vee\square}$ if $x'_{3}\in\Gamma_{0}$ and
a $D_{(\bigcirc,n-1)\vee\bot}$ diagram else. 
\item $((\bigcirc,n)\vee(\bigcirc,m))$ with $(n,m)\neq(3,3)$. Then $\mu_{k}(\Gamma)$
is a \linebreak $D_{(\bigcirc,n-1),(\bigcirc,m-1)}$ diagram of width
zero. 
\item $(\bigcirc,n)\vee\boxslash$. Then $\mu_{k}(\Gamma)$ is a $D_{(\bigcirc,n+1)\wedge\boxslash}$
diagram. 
\item $(\bigcirc,n)\vee\square$. Then $\mu_{k}(\Gamma)$ is a $D_{(\bigcirc,n+1)\wedge(\bigcirc,3)}$
diagram. 
\item $(\bigcirc,n)\vee\bot$. Then $\mu_{k}(\Gamma)$ is a $D_{(\bigcirc,n+1)\wedge\boxslash}$
diagram if $\{a'_{1},a'_{2},\bullet\}$ carries non-linear orientation
and a $D_{(\bigcirc,n+1)\wedge(\bigcirc,3)}$ diagram else. 
\item $(\boxslash\vee\boxslash)$. Then assuming without loss of generality
that $x_{2}=\bullet=x'_{1}$, $\mu_{k}(\Gamma)$ is a $D_{\boxslash\vee\boxslash}$
diagram if $\{x_{1},\bullet,x'_{2}\}$ carries non-linear orientation
and a $D_{\boxtimes}$ diagram else. 
\item $(\boxslash\vee\square)$. Then $\mu_{k}(\Gamma)$ is a $D_{(\bigcirc,3)\wedge\boxslash}$
diagram. 
\item $(\square\vee\square)$. Then $\mu_{k}(\Gamma)$ is a $D_{(\bigcirc,3)\vee(\bigcirc,3)}$
diagram. 
\item $(\boxslash\vee\bot)$. Then $\mu_{k}(\Gamma)$ is a $D_{\boxslash\vee\bot}$
diagram if one of $\mbox{deg}{}^{\pm}(\bullet)$ equals four. Else
one of $\mbox{\emph{\emph{deg}}}{}^{\pm}(\bullet)$ equals three and
$\mu_{k}$ yields a $D_{\boxtimes}$ diagram if $\{a'_{1},a'_{2},\bullet\}$
carries non-linear orientation and $D_{(\bigcirc,3)\wedge\boxslash}$
diagram else. 
\item $(\square\vee\bot)$. Then $\mu_{k}(\Gamma)$ is a $D_{(\bigcirc,3)\wedge\boxslash}$
diagram if $\{a'_{1},a'_{2},\bullet\}$ carries non-linear orientation
and $D_{(\bigcirc,3)\vee(\bigcirc,3)}$ diagram else. 
\item $(\bot\vee\bot)$. Then $\mu_{k}(\Gamma)$ is a $D_{\bot\vee\bot}$
diagram if one of $\mbox{deg}^{\pm}(\bullet)$ equals four, a $D_{(\bigcirc,3)\wedge\boxslash}$
diagram if one of $\mbox{deg}{}^{\pm}(\bullet)$ equals three and
a $D_{\boxtimes}$ diagram else.\\

\end{enumerate}
\end{enumerate}
\item Let $\Gamma$ be a $D_{(\bigcirc,n)\wedge\star}$ diagram, where $\star\in\{\square,\boxslash,\overleftrightarrow{\boxslash}\}$.

\begin{enumerate}
\item If $k\notin\left((\bigcirc,n)\wedge\star\right)_{0}$ then $\mu_{k}(\Gamma)$
is a $D_{(\bigcirc,n)\wedge\star}$ diagram. 
\item If $k=x_{i}$ for some $i$ then $\mu_{k}(\Gamma)$ is a $D_{(\bigcirc,n+1)\wedge\star'}$
diagram. 
\item If $k\in\left((\bigcirc,n)\wedge\star\right)_{0}\setminus\{\bullet_{j},a'_{j}\}_{1\leq j\leq2}$
and $n>3$, then $\mu_{k}(\Gamma)$ is a $D_{(\bigcirc,n-1)\wedge\star}$
diagram. 
\item If $k\in\left((\bigcirc,3)\wedge\star\right)_{0}\setminus\{\bullet_{j},a'_{j}\}_{1\leq j\leq2}$,
then $\mu_{k}(\Gamma)$ is a $D_{\boxslash\wedge\boxslash}$ diagram
if $\star=\overleftrightarrow{\boxslash}$. Else $\star=\boxslash$
and $\mu_{k}$ yields a $D_{\square\wedge\square}$ diagram. 
\item Else $k\in\{\bullet_{j},a'_{j}\}_{1\leq j\leq2}$. Assume $\star$
equals

\begin{enumerate}
\item $\overleftrightarrow{\boxslash}$. If $k\in\{\bullet_{1},\bullet_{2}\}$
then $\mu_{k}(\Gamma)$ is a $D_{(\bigcirc,n)\wedge\overleftrightarrow{\boxslash}}$
diagram. Else $k\in\{a'_{1},a'_{2}\}$ and $\mu_{k}$ yields a $D_{(\bigcirc,n+1)\wedge\square}$
diagram. 
\item $\boxslash$ . If $k\in\{a'_{1},a'_{2}\}$ then $\mu_{k}(\Gamma)$
is a $D_{(\bigcirc,n+1)\wedge\square}$ diagram. Else $k\in\{\bullet_{1},\bullet_{2}\}$
and

\begin{enumerate}
\item for $n>3$, $\mu_{k}(\Gamma)$ is a $D_{(\bigcirc,n-1)\vee\boxslash}$
diagram of width zero if $\mbox{deg}(k)=5$ and a $D_{(\bigcirc,n-1)\vee\bot}$
of width zero else. 
\item for $n=3$, $\mu_{k}(\Gamma)$ is a $D_{\boxslash\wedge\square}$
diagram if $\{x_{1},x_{2},a'_{1}\}\subset\Gamma$. Else $\mid\{x_{1},x_{2},a'_{1}\}\cap\Gamma\mid=2$
and $\mu_{k}(\Gamma)$ is a\\
 $(\alpha)$ $D_{\boxslash\vee\bot}$ of width zero if $\mbox{deg}(k)=5$\emph{.}
\\
 $(\beta)$ $D_{\square\vee\bot}$ of width zero else. 
\end{enumerate}
\item $\square$.

\begin{enumerate}
\item If $n>4$ then $\mu_{k}(\Gamma)$ is a \\
 $(\alpha)$ $D_{(\bigcirc,n-1)\vee(\bigcirc,3)}$ diagram if $k\in\{\bullet_{1},\bullet_{2}\}$.
Moreover if $\mbox{deg}(k)=3$ then $x'_{3}\notin(\mu_{k}(\Gamma))_{0}$
.\\
 $(\beta)$ Else $k\in\{a'_{1},a'_{2}\}$ and $\mu_{a'_{1}}(\Gamma)$
is a $D_{(\bigcirc,n-1)\wedge\boxslash}$ diagram while $\mu_{a'_{2}}(\Gamma)$
is a $D_{(\bigcirc,n-1)\wedge\overleftrightarrow{\boxslash}}$ diagram. 
\item If $n=3$ then $\mu_{k}(\Gamma)$ is a \\
 $(\alpha)$ $D_{(\bigcirc,3)\wedge\square}$ diagram if $k\in\{\bullet_{1},\bullet_{2}\}$.\\
 $(\beta)$ Else $k\in\{a'_{1},a'_{2}\}$ and $\mu_{a'_{1}}(\Gamma)$
is a $D_{\square\wedge\square}$ diagram while $\mu_{a'_{2}}(\Gamma)$
is a $D_{\boxslash\wedge\boxslash}$ diagram.\\

\end{enumerate}
\end{enumerate}
\end{enumerate}
\item Let $\Gamma$ be a $D_{\star\wedge\star}$ diagram, where $\star\in\{\square,\boxslash\}$.

\begin{enumerate}
\item If $k\notin\left(\star\wedge\star\right)_{0}$ then $\mu_{k}(\Gamma)$
is a $D_{\star\wedge\star}$ diagram. 
\item Else $k\in\left(\star\wedge\star\right)_{0}$.

\begin{enumerate}
\item If $k=\bullet$ then $\mu_{k}(\Gamma)$ is a $D_{(\bigcirc,3)\wedge\boxslash}$
diagram if $\star=\square$ and $D_{(\bigcirc,3)\wedge\overleftrightarrow{\boxslash}}$
diagram else. 
\item If $k\in\{x_{1},x_{2}\}$ then $\mu_{k}(\Gamma)$ is a $D_{\boxtimes}$
diagram if $\star=\square$ and a $D_{\boxslash\wedge\boxslash}$
diagram else. 
\item If $k\in\{a_{1},a_{2}\}$ then $\mu_{k}(\Gamma)$ is a $D_{(\bigcirc,3)\wedge\square}$
diagram.\\

\end{enumerate}
\end{enumerate}
\item Let $\Gamma$ be a $D_{\boxtimes}$ diagram.

\begin{enumerate}
\item If $k\notin\boxtimes_{0}$ then $\mu_{k}(\Gamma)$ is a $D_{\boxtimes}$
diagram.\\
 \\
 Else $k\in\boxtimes_{0}$. 
\item If $k\in\{a_{i}\}_{1\leq i\leq4}$ then $\mu_{k}(\Gamma)$ is a $D_{\square\wedge\square}$
diagram. 
\item If $k=x_{1}$ then $\mu_{k}(\Gamma)$ is a

\begin{enumerate}
\item $D_{\boxslash\vee\boxslash}$ diagram if $\mbox{deg}(x)=6$\emph{.} 
\item $D_{\boxslash\vee\bot}$ diagram if $\mbox{deg}(x)=5$\emph{.} 
\item \emph{$D_{\bot\vee\bot}$}  diagram else.\\

\end{enumerate}
\end{enumerate}
\end{enumerate}
The last part of Theorem \ref{Theorem (mutation classes)} is proved
in the same way as part (4). The difficult part is to show that any
diagram belonging to one of the families defined in subsection \ref{sub:The--diagrams for D^(1)}
is of mutation type $\mathbb{D}$. This is done by excluding one by
one the cases to be considered, the essential point being that we
have a complete overview of the effect of mutation on these diagrams
by Fact \ref{lem: D^(1) mutation steps}. As for types $\mathbb{D}$
and $\mathbb{B}^{(1)}$ the aim is to reduce to the case of a 'maximal
cycle' and then to apply Lemma \ref{Lemma shrink cycle}.

\begin{thm*}
\textbf{\emph{\ref{Theorem (mutation classes)}}} \emph{(6)} Let $\Gamma$
be a connected diagram. Then $\Gamma$ is of mutation type $\mathbb{D}^{(1)}$
if and only if it is a $D_{\star,\star'},D_{\star\vee\star',}D_{\star\wedge\star'}$
diagram, where $\star,\star'\in\bigl\{(\bigcirc,n)_{n\geq3},\square,\boxslash,\bot,\overleftrightarrow{\boxslash}\bigr\}$,
or a $D_{\boxtimes}$ diagram. 
\end{thm*}
\begin{proof}
Any diagram with underlying graph of Dynkin type $\mathbb{D}^{(1)}$
is a $D_{\bot,\bot}$ diagram, hence Fact \ref{lem: D^(1) mutation steps}
implies that any diagram of mutation type $\mathbb{D}^{(1)}$ is a
$D_{\star,\star'},D_{\star\vee\star',}D_{\star\wedge\star'}$ diagram,
where $\star,\star'\in\bigl\{(\bigcirc,n)_{n\geq3},\square,\boxslash,\bot,\overleftrightarrow{\boxslash}\bigr\}$,
or a $D_{\boxtimes}$ diagram.

To show the other direction, first assume that $|\Gamma_{0}|=5$.
This restricts the possibilities for $\Gamma$ to $D_{\star\wedge\star'}$,
where $\star,\star'\in\bigl\{(\bigcirc,3),\square,\boxslash,\overleftrightarrow{\boxslash}\bigr\}$,
or $D_{\boxtimes}$, or $D_{\bot\vee\bot}$ diagrams. Note that the
last mentioned has underlying undirected graph $D_{4}^{(1)}$ already.
Moreover Fact \ref{lem: D^(1) mutation steps} (5) (c) (iii) implies
that $D_{\boxtimes}$ diagrams are also of mutation type $\mathbb{D}^{(1)}$.
Hence we only have to deal with $D_{\star\wedge\star'}$ diagrams,
where $\star,\star'\in\bigl\{(\bigcirc,3),\square,\boxslash,\overleftrightarrow{\boxslash}\bigr\}$.
By Fact \ref{lem: D^(1) mutation steps} (4) (b) (iii) we may restrict
to $D_{(\bigcirc,3)\wedge\star}$ diagrams, where $\star\in\{\square,\boxslash,\overleftrightarrow{\boxslash}\bigr\}$.
Now (3) (e) (i) and (ii) of the same Fact imply that it suffices to
deal with $D_{(\bigcirc,3)\wedge\square}$ diagrams. Using (3) (e)
(iii) (B) ($\beta)$ we see that we may restrict to $D_{\square\wedge\square}$
diagrams after all. But then (4) (b) (ii) yields that all these diagrams
are mutation equivalent to $D_{\boxtimes}$ diagrams, which we already
know to be of mutation type $\mathbb{D}^{(1)}$. This finishes the
proof for $|\Gamma_{0}|=5$

For $|\Gamma_{0}|>5$ we apply the results of Fact \ref{lem: D^(1) mutation steps}
in order to reduce to the case where $\Gamma$ is a $D_{(\bigcirc,m)\wedge\square}$
diagram, where $m\leq|\Gamma_{0}|-2$. First note that any $D_{\star,\star'}$
diagram, where $\star,\star'\in\bigl\{(\bigcirc,m)_{m\geq3},\square,\boxslash,\bot\bigr\}$,
is mutation equivalent to a $D_{\bar{\star},\bar{\star'}}$ diagram,
where $\bar{\star},\bar{\star'}\in\bigl\{(\bigcirc,m)_{m\geq3},\square,\boxslash,\bot\bigr\}$,
of width zero. Hence (1) (c) and (5) (b) reduce to the case of $D_{\star\vee\star',}D_{\star\wedge\star'}$
diagrams with $\star,\star'\in\bigl\{(\bigcirc,m)_{m\geq3},\square,\boxslash,\bot,\overleftrightarrow{\boxslash}\bigr\}$;
here we use that for $D_{\boxslash,\star'}$ and $D_{\bot,\star}$
diagrams of width zero we may, by subsequent mutation in $a_{1}$
and $a_{2}$, always assume to be in case (1) (c) (iii) (A) respectively
(1) (c) (iv) (A) of Fact \ref{lem: D^(1) mutation steps}. Next, (2)
(b) reduces to $D_{\star\wedge\star'}$ diagrams with $\star,\star'\in\bigl\{(\bigcirc,m)_{m\geq3},\square,\boxslash,\overleftrightarrow{\boxslash}\bigr\}$:
this is immediate for $(\bigcirc,n)\vee\star'$, $\star'\neq(\bigcirc,m)$,
and $(\boxslash\vee\square)$. For $(\star\vee\bot)$, $\star\in\{\square,\boxslash,\bot\}$,
use that subsequent mutation in $a_{1}$ and $a_{2}$ allows for choosing
any of the cases (2) (b) (x), (xi) or (xii). If $(\star\vee\star')\in\{(\square\vee\square),(\boxslash\vee\boxslash)\}$
then mutation in $a_{1}$ yields a $D_{(\boxslash\vee\square)}$ diagram.
For $(\star\vee\star')=((\bigcirc,n),(\bigcirc,m))$ use (2) (a) (ii)
$n-2$ times. Hence we have reduced to $D_{\star\wedge\star'}$ diagrams,
where $\star,\star'\in\bigl\{(\bigcirc,m)_{m\geq3},\square,\boxslash,\overleftrightarrow{\boxslash}\bigr\}$.
Now apply Fact \ref{lem: D^(1) mutation steps} (4) (b) (iii) to exclude
$\star,\star'\in\{\boxslash,\square\}$ and (3) (e) (i) as well as
(ii) to get rid of $\star\wedge\star'\in\left\{ (\bigcirc,m)\wedge\boxslash,(\bigcirc,m)\wedge\overleftrightarrow{\boxslash}\right\} $.

Thus we may assume that $\Gamma$ is a $D_{(\bigcirc,m)\wedge\square}$
diagram, where $m\leq|\Gamma_{0}|-2$. Set $n:=|\Gamma_{0}|$ and
apply Fact \ref{lem: D^(1) mutation steps} (3) (b) until $\Gamma$
is a $D_{(\bigcirc,n-2)\wedge\square}$ diagram. Due to cardinality
Fact \ref{lem: D^(1) mutation steps} (3) (e) (iii) (A) $(\alpha)$
yields that we may assume $\Gamma$ to be of the form as described
in the second case of (2) (b) (ii) of the same Fact; hence we are
in the situation of Lemma \ref{Lemma shrink cycle} with $\Gamma'$
a diagram with underlying undirected graph $A_{3}$, having $y$ as
its middle vertex. Thus $\Gamma$ is of mutation type $\mathbb{D}^{(1)}$. 
\end{proof}
\bibliographystyle{amsplain}
\bibliography{mutation classes}

\end{document}